\documentclass[12pt,a4paper]{article}
\usepackage{latexsym,amssymb,amsfonts,amsmath,amsthm,nccmath,cases}
\usepackage{amsmath,amsthm,fullpage,mdwlist}
\usepackage{pdfsync}
\usepackage{amssymb, amsfonts, graphicx, mathrsfs,hyperref}
\usepackage[usenames,dvipsnames]{xcolor}
\usepackage{enumitem}
\usepackage[margin=2cm]{geometry}
\usepackage{multicol}

\numberwithin{equation}{section}
\newtheorem{theorem}{Theorem}[section]
\newtheorem{corollary}[theorem]{Corollary}
\newtheorem{lemma}[theorem]{Lemma}

\theoremstyle{definition}
\newtheorem{definition}[theorem]{Definition}

\newtheorem{conjecture}[theorem]{Conjecture}

\def\OA{\mathrm{OA}}

\def\eref#1{$(\ref{#1})$}
\def\sref#1{\S$\ref{#1}$}
\def\lref#1{Lemma~$\ref{#1}$}
\def\tref#1{Theorem~$\ref{#1}$}
\def\Tref#1{Table~$\ref{#1}$}
\def\cjref#1{Conjecture~$\ref{#1}$}
\def\cyref#1{Corollary~$\ref{#1}$}

\def \L{\mathcal{L}}
\def \M{\mathscr{M}}
\def \A{\mathscr{A}}
\def \G{\mathscr{G}}

\def \dim {\mathcal{B}}
\def \blocks {\mathscr{B}}
\def \sym {\mathcal{S}}
\def \Z {\mathbb{Z}}
\def \= {\equiv}

\def \binom#1#2{{#1\choose#2}}

\def \leq {\leqslant}
\def \geq {\geqslant}
\def \le {\leqslant}
\def \ge {\geqslant}

\def \eps{\varepsilon}
\def \mod#1{{\:({\rm mod}\ #1)}}

\let\oldproofname=\proofname
\renewcommand{\proofname}{\rm\bf{\oldproofname}}

\def\eps{\varepsilon}

\newcommand{\ignore}[1]{}

\title{Parity of Sets of Mutually Orthogonal Latin Squares \footnote{Research
supported by ARC grant DP150100506}}

\author{Nevena Franceti\'{c}\footnote{
School of Mathematical Sciences, Monash University, Victoria 3800, Australia} 
\quad Sarada Herke\footnotemark[2] \footnote{Current address: School of Mathematics and Physics, The University of Queensland,
QLD 4072, Australia}  \quad Ian M. Wanless\footnotemark[2]}
\date{}

\begin{document}

\maketitle

\begin{abstract}
Every Latin square has three attributes that can be even or odd, but
any two of these attributes determines the third. Hence the parity of
a Latin square has an information content of 2 bits.  We extend the
definition of parity from Latin squares to sets of mutually orthogonal
Latin squares (MOLS) and the corresponding orthogonal arrays (OA).
Suppose the parity of an $\OA(k,n)$ has an information content of
$\dim(k,n)$ bits.  We show that $\dim(k,n) \leq {k \choose 2}-1$. For
the case corresponding to projective planes we prove a tighter bound,
namely $\dim(n+1,n) \leq {n \choose 2}$ when $n$ is odd and
$\dim(n+1,n) \leq {n \choose 2}-1$ when $n$ is even.  Using the
existence of MOLS with subMOLS, we prove that if 
$\dim(k,n)={k \choose 2}-1$ then $\dim(k,N) = {k \choose 2}-1$ 
for all sufficiently large $N$.

Let the {\em ensemble} of an $\OA$ be the set of Latin squares derived
by interpreting any three columns of the $\OA$ as a Latin square. We
demonstrate many restrictions on the number of Latin squares of each
parity that the ensemble of an $\OA(k,n)$ can contain. These
restrictions depend on $n\mod4$ and give some insight as to why it is
harder to build projective planes of order $n\= 2\mod 4$ than for
$n\not\= 2\mod4$. For example, we prove that when $n\= 2\mod 4$ it is
impossible to build an $\OA(n+1,n)$ for which all Latin squares in the
ensemble are isotopic (equivalent to each other up to
permutation of the rows, columns and symbols).

\bigskip

\end{abstract}

\section{Introduction}\label{s:intro}

The aim of this paper is to develop a notion of parity for the
orthogonal arrays that define sets of mutually orthogonal Latin
squares (MOLS).  The important notion of parity for permutations is
widely known. A \emph{Latin square} of order $n$ is an $n \times n$
square with entries in an $n$-set $\Lambda$, called the alphabet,
having the property that every element of $\Lambda$ occurs exactly
once in each row and each column of the square.  Latin squares are
$2$-dimensional analogues of permutations and they too have a notion
of parity. This parity plays a pivotal role in a famous conjecture of
Alon and Tarsi (see e.g.~\cite{AL15,SW12} and the references therein) and has
also proved crucial in a variety of other quite distinct
investigations.  In \cite{DGGL10}, parity was found to explain
observed limitations on which Latin squares could be embedded together
in topological surfaces.  In \cite{Wan04} and \cite{KMOW14}, parity
explains large components that arise in graphs made by local
switchings in Latin squares or in 1-factorisations of the complete
graph, respectively. Parity considerations can also assist in 
diagnosing symmetries of Latin squares \cite{Kot12}.

It is clear that parity of single Latin squares is a useful
concept. It is therefore natural to try to extend this concept to sets
of MOLS.  A set of MOLS is a set of Latin squares such that when any
two of the squares are superimposed, every ordered pair of symbols 
occurs exactly once. For any list $M=[M_1,\dots,M_{k-2}]$ of MOLS
on an alphabet $\Lambda$ we define an $n^2\times k$ matrix, denoted
$\A(M)$, by taking one row
$\big[r,c,M_1[r,c],\dots,M_{k-2}[r,c]\big]$ for each pair $(r,c)\in\Lambda^2$.
(For the sake of definiteness, we insist that these rows are ordered
lexicographically. Also, if $M$ is given as a set rather than a list, then
we impose lexicographic order on $M$ in order to create $\A(M)$.) 
Now, $\A(M)$ is an \emph{orthogonal array}
$\OA(k,n)$ (of strength $2$ with $n$ levels and index $1$) because it
has the defining property that every pair of columns contains
every ordered pair of elements of $\Lambda$ exactly once.  Conversely,
if $A$ is an $\OA(k,n)$ we define $\M(A)$ to be the set of $k-2$ MOLS
formed by taking the entry in row $r$, column $c$ of the $i$-th Latin
square to be the entry in column $i+2$ of the row of $A$ that begins
$[r,c,\dots]$. In this sense, an $\OA(k,n)$ is equivalent to a set of
$k-2$ MOLS of order $n$ (see e.g.~\cite[III.3]{handbook} for more
details and background).  Throughout this paper, we assume that $k$
and $n$ are integers and that $n\geq 2$ and $3 \leq k \leq n+1$.

Let $\sym_\Lambda$ denote the permutations of $\Lambda$.
Two orthogonal arrays on alphabet $\Lambda$ are {\em isotopic} if one
can be obtained from the other by some sequence of operations of the
following type: choose a $\gamma\in\sym_\Lambda$ and a column $c$ and
apply $\gamma$ to every entry in column $c$. We say two orthogonal
arrays are {\em conjugate} if one can be obtained from the other by
permuting the columns. We say two Latin squares
are isotopic (respectively, conjugate) if their
orthogonal arrays are isotopic (respectively, conjugate).
We say two orthogonal arrays are isomorphic if, 
up to possible reordering of the rows of the arrays,
one is isotopic to a conjugate of the other. 


A finite projective plane of order $n$ (see e.g.~\cite{handbook} for
the definition) can be used to define an $\OA(n+1,n)$ and vice
versa. However, there are some subtleties to this relationship.  Let
$\L$ be a line of a finite projective plane of order $n$. We can make
an $\OA(n+1,n)$, using $\L$ as the {\em ``line at infinity''}, as
follows.  First, we number the points on $\L$ as $p_1,\ldots,p_{n+1}$.
Next, for $1\le i\le n+1$, we number the lines (other than $\L$)
through $p_i$, calling them $\ell_{i1},\ldots,\ell_{in}$.  Finally,
for each point $q$ not on $\L$ we add a row to $A$ which has entry $i$
in column $j$ if the line through $q$ and $p_j$ is $\ell_{ji}$.  The
choices for the numbering of lines and points do not change the
isomorphism class of the resulting orthogonal array.  However,
different choices for $\L$ may produce non-isomorphic orthogonal
arrays.  Hence, each projective plane of order $n$ can potentially
produce representatives from up to $n^2+n+1$ isomorphism classes of
$\OA(n+1,n)$. The reverse relationship is simpler. Every $\OA(n+1,n)$
can be derived in the above way from a unique projective plane.


Given a Latin square $L=(l_{ij})$ of order $n$, we can identify $3$
parities $\pi_r(L)$, $\pi_c(L)$ and $\pi_s(L)$ as follows.  We assume
that the symbols $\Lambda$ index the rows and columns of $L$ (in this
paper $\Lambda$ will either be $\{1,\dots,n\}$ or $\Z_n$).  Let
$\pi:\sym_\Lambda\rightarrow\Z_2$ denote the usual {\em parity}
homomorphism with kernel the alternating group.  For all $i\in\Lambda$
we can define a permutation of $\Lambda$ by $j\mapsto l_{ij}$.
Applying $\pi$ to these permutations and taking the sum, mod 2, we
obtain the {\em row-parity} $\pi_r(L)$. The {\em column-parity}
$\pi_c(L)$ is defined similarly, using each permutation $i\mapsto
l_{ij}$ formed by fixing some $j\in\Lambda$.  The {\em symbol-parity}
$\pi_s(L)$ is the sum of the parities of the permutations formed by
fixing an $\ell\in\Lambda$ and mapping $i\mapsto j$ whenever
$l_{ij}=\ell$. These three parities are related by
\begin{equation}\label{e:fundLSpar}
\pi_r + \pi_c + \pi_s \= {n \choose 2}\quad\mod2.
\end{equation} 
This relation has been rediscovered numerous times.  Different proofs
have been published in \cite{DGGL10, Gly10, Jan95, Wan04, Zappa96} and
we are also aware of other researchers finding their own proof but not
publishing it.  We will prove a generalisation of \eref{e:fundLSpar}
in \lref{l:LStriangleparity}, providing a new proof of the
original result in the process.

By \eref{e:fundLSpar}, a Latin square has one of the following 
{\em parity types}:
\begin{align*}
\pi_r \pi_c \pi_s  &\in \{000, 011, 101, 110 \} 
  && \textrm{ if } n \equiv 0,1 \mod{4},\\
\pi_r \pi_c \pi_s  &\in \{111, 100, 010, 001 \} 
  && \textrm{ if } n \equiv 2,3 \mod{4}.
\end{align*}
It is known \cite{CW16} that the proportion of Latin squares 
which have each of the four parity types that are possible
for order $n$ is $\frac14+o(1)$ as $n\rightarrow\infty$
(for a related result, see \cite{Alp17}). 
An \textit{equiparity} Latin square has $\pi_r \pi_c \pi_s \in \{000, 111\}$. 
Loosely speaking, ``nice'' Latin squares have a greater than
$1/4$ chance of being equiparity.
For example, all Cayley tables of finite groups are necessarily equiparity, 
because they are isotopic to all of their conjugates \cite[Thm~4.2.2]{KD15}.


In \sref{s:definitions}, we will define two notions of parity
associated with an orthogonal array. The first one we refer to as
$\tau$-parity; it is a direct generalisation of the definition of
parities $\pi_r,\pi_c,\pi_s$ for a Latin square. The second one we
call $\sigma$-parity. It was introduced by Glynn and Byatt~\cite{GB12}
(see also Glynn~\cite{Gly10}) for $n$ even, but we define it for all
$n$. We establish the precise relationship between these two seemingly
different definitions of parity of an $\OA$. For the remainder of the
paper we look into properties of $\tau$-parity and $\sigma$-parity and
insights that each can offer.

In \sref{s:graphs} we introduce graphs that are useful tools for
studying both notions of parity.  
In \sref{s:numpar} we consider the question of how many different
$\tau$-parities an $\OA(k,n)$ can have. We view this question in
information theory terms, by measuring the number of independent bits
of information that there are in the $\tau$-parity. For example,
the $\tau$-parity of an $\OA(3,n)$ has an information content of 2 bits
given that any two of $\pi_r,\pi_c,\pi_s$ determine the third, by
\eref{e:fundLSpar}. We prove a bound on the information content
of the $\tau$-parity of an $\OA(k,n)$, and prove that this bound is
achieved for all large $n$ when $k\le 5$.

In \sref{s:PP} we consider the important special case of the
orthogonal arrays derived from projective planes as described above.
The theory developed in earlier sections is applied to this case.
Stronger conclusions can be drawn than in the general case. For example,
we show that the information content in the $\tau$-parity is lower than
the general bound from \sref{s:numpar}.

Let $A$ be an $\OA(k,n)$. The {\it ensemble} of $A$ is the set of
Latin squares $\M(B)$, where $B$ ranges across all ${k\choose 3}$
choices of $\OA(3,n)$ formed by $3$ columns of $A$ (the columns should
occur in the same order in $B$ as they do in $A$).  In
\sref{s:ensemble} we investigate the parities of the Latin squares in
the ensemble of $A$. We find significant restrictions on the number of
equiparity Latin squares, particularly in the case when $k=n+1$.
Among other things, these allow us to deduce that for $n\= 2\mod4$
there is no $\OA(n+1,n)$ for which all Latin squares in the ensemble
are isotopic to each other. This is in contrast with the 
$n\not\= 2\mod4$ case, where the Desarguesian projective planes
provide examples of $\OA(n+1,n)$ for which all Latin squares in the
ensemble are isotopic to the Cayley table of the elementary abelian
group of order $n$.

\section{Two notions of parity}\label{s:definitions}

Throughout the paper, we use the discrete interval notation
$[a,b]=\{a,a+1,\dots, b\}$ for $a,b \in \Z$.  In this section we
define two notions of parity for orthogonal arrays and demonstrate the
relationship between them.  The two notions of parity are called
$\tau$-parity and $\sigma$-parity.  Each $\tau^c_{ij}$ of the
$\tau$-parity and each $\pi(\sigma_{ij})$ of the $\sigma$-parity is an
element of $\Z_2$ and hence, throughout this paper, every equation
involving $\tau$-parities or $\sigma$-parities is assumed to be
calculated in $\Z_2$. Also, because it makes the definition of
$\sigma$-parity more natural, we will choose to always index the rows
of an $\OA$ by $\Lambda^2$, where $\Lambda$ is the alphabet.

\subsection{$\tau$-parity}\label{ss:tau-parities}

\begin{definition}[$\tau$-parity]
Let $A=(a_{r\ell})$ be an $\OA(k,n)$ on an alphabet $\Lambda$, where
$r \in \Lambda^2$ and $\ell \in [1,k]$.  For each ordered triple
$(c,i,j)$ of distinct numbers in $[1,k]$ and for each $s \in
\Lambda$, define $\rho_s(c,i,j)$ to be the permutation of $\Lambda$
mapping $a_{ri} \mapsto a_{rj}$ whenever $a_{rc}=s$.  Define
$\tau^c_{ij}=\tau^c_{ij}(A)$ as the sum of the parities of these $n$
permutations; that is, 
$\tau^c_{ij}=\sum_{s \in \Lambda}^{}\pi(\rho_s(c,i,j))$.  
We refer to the vector of parities $\tau^c_{ij}$ indexed by the
$k(k-1)(k-2)$ triples $(c,i,j)$ as the \textit{$\tau$-parity of $A$}.
\end{definition}

Note that $\{a_{ri}:a_{rc}=s \} = \Lambda = \{a_{rj}:a_{rc}=s\}$ 
by the definition of an orthogonal array. Hence,
$\{\rho_s(c,i,j) :  s \in \Lambda\}$ is indeed a set of $n$
well-defined permutations.

The $\tau$-parity of an orthogonal array $A$ naturally extends the
notion of the parity of a single Latin square to a set of MOLS. If
$M_1,\dots,M_{k-2}$ are the MOLS in $\M(A)$ then for the Latin square
$M_{i-2}$ we have $\pi_r =\tau^1_{2i}$, $\pi_c=\tau^2_{1i}$, and
$\pi_s = \tau^i_{12}$.  For distinct $c,i,j, \ell \in [1,k]$ we have
\begin{align}
 \tau_{ij}^c  &= \tau_{ji}^c, \label{e:tauidxcommutativity} \\
 \tau_{ij}^c  &= \tau_{i\ell}^c + \tau_{\ell j}^c. \label{e:Fixedcol}
\end{align}
The first equation follows from the fact that a permutation and
its inverse have the same parity. The second equation follows from
composition of permutations.

Next we consider $\tau$-parities of isomorphic orthogonal arrays.  
Permuting the rows of an orthogonal array $A$ has no effect on $\M(A)$,
and permuting the symbols within a column results in an isotopic set of 
MOLS. Permuting the columns yields a conjugate set of MOLS.
We investigate the effect of these basic operations on the
$\tau$-parity.

\begin{lemma}\label{l:tauparconj}
Let $A_1$ and $A_2$ be two orthogonal arrays, $\OA(k,n)$, on alphabet
$\Lambda$.  Index the columns of $A_1$ and $A_2$ by $[1,k]$. 
Let $c,i,j \in [1,k]$ be distinct integers.
\begin{itemize}
\item[(i)] If $A_2$ is obtained by permuting the rows of $A_1$, then
  $\tau^c_{ij}(A_2)=\tau^c_{ij}(A_1)$.
\item[(ii)] If $A_2$ is obtained by permuting the columns of $A_1$ by 
  $\gamma$, then $\tau^{\gamma(c)}_{\gamma(i)\gamma(j)}(A_2)=\tau^c_{ij}(A_1)$.
\item[(iii)] Let $\gamma\in\sym_\Lambda$ and $c\in[1,k]$.  If $A_2$ is
  obtained by applying $\gamma$ to every entry in column $c$ of $A_1$,
  then
\begin{align*}
\tau^{c}_{ij}(A_2) &= \tau^{c}_{ij}(A_1), \\
\tau^i_{jc}(A_2) &= \tau^i_{jc}(A_1) + n \cdot \pi(\gamma) ,\\
\tau^{d}_{ij}(A_2) &= \tau^{d}_{ij}(A_1) 
\textrm{ if } d\notin\{c,i,j\}.
\end{align*}
\end{itemize}
\end{lemma}

\begin{proof}
Statements (i) and (ii) follow immediately from the definition of
$\tau$-parity.  To prove statement (iii), first note that
$\tau^{c}_{ij}(A_1)$ and $\tau^{c}_{ij}(A_2)$ are each the sum of
the parities of the same set of $n$ permutations of $\Lambda$, hence
$\tau^{c}_{ij}(A_2)=\tau^{c}_{ij}(A_1)$.  It is also clear that
$\tau^{d}_{ij}(A_2) = \tau^{d}_{ij}(A_1)$ for $d\notin\{c,i,j\}$. 
Finally, in $\Z_2$ we have
\[
\tau^i_{jc}(A_2) 
= \sum_{s \in \Lambda} \pi(\gamma \rho_s(i,j,c)) 
= \sum_{s \in \Lambda} \pi(\rho_s(i,j,c)) + \sum_{s \in \Lambda}\pi(\gamma)
= \tau^i_{jc}(A_1) + n\cdot\pi(\gamma).
\qedhere
\]
\end{proof}

Suppose $M=[M_1,M_2,\dots,M_{k}]$ are MOLS and that $A=\A(M)$ is the
corresponding $\OA(k+2,n)$. 
Let $A'$ be the $\OA(k+2,n)$ obtained by permuting the columns of $A$ by
some permutation $\gamma$. Suppose $M'=\M(A')=[M'_1,M'_2,\dots,M'_{k}]$.
Let $\pi_{r,i},\pi_{c,i},\pi_{s,i}$ denote,
respectively, the row, column and symbol parity of $M_i$ for $i\in[1,k]$,
and let $\pi'_{r,i},\pi'_{c,i},\pi'_{s,i}$ be the corresponding parities
of $M'_i$. It will follow from our work in \sref{s:numpar} that there
is not enough information in $\{\pi_{r,i},\pi_{c,i},\pi_{s,i}:i\in[1,k]\}$ 
to determine $\{\pi'_{r,i},\pi'_{c,i},\pi'_{s,i}:i\in[1,k]\}$ for all
$\gamma$. However, there are some things we can say:
\begin{itemize}
\item[(i)]
If the first two columns are fixed points of $\gamma$ then $M'$ is simply
a reordering of the Latin squares in $M$, and the parities will be permuted
accordingly. 

\item[(ii)]
If $\gamma$ is the transposition $(12)$ then $M'_i$ is the transpose of
$M_i$, so $\pi'_{r,i}=\pi_{c,i}$, $\pi'_{c,i}=\pi_{r,i}$ and $\pi'_{s,i}=\pi_{s,i}$
for $i\in[1,k]$.

\item[(iii)] 
If $\gamma$ is the transposition $(23)$ then $M'_1$ is a conjugate of
$M_1$ for which $\pi'_{r,1}=\pi_{r,1}$, $\pi'_{c,1}=\pi_{s,1}$ and 
$\pi'_{s,1}=\pi_{c,1}$. Moreover, $\pi'_{r,i}=\pi_{r,i}+\pi_{r,1}$
for $i\in[2,k]$, because $\tau^1_{2(i+2)}=\tau^1_{23}+\tau^1_{3(i+2)}$ by
\eref{e:Fixedcol}.

\end{itemize}

Note that any permutation of the columns can be achieved by composing the
permutations in (i), (ii) and (iii) above.

\subsection{$\sigma$-parity}\label{ss:sigma-parities}

In this subsection, we consider an alternative definition of parity of
an orthogonal array. This definition was introduced by Glynn and
Byatt~\cite{GB12} (see also Glynn~\cite{Gly10}) for orthogonal arrays
with even alphabet size, although it is also useful for odd alphabet
sizes.

\begin{definition}[$\sigma$-parity] 
Let $A=(a_{r\ell})$ be an $\OA(k,n)$ on alphabet $\Lambda$, where 
$r\in \Lambda^2$ and $\ell \in [1,k]$. Let $i,j \in [1,k]$ be two
distinct integers. Then $\sigma_{ij}:\Lambda^2 \rightarrow \Lambda^2$
is the permutation defined by $\sigma_{ij}(r) = (a_{ri},a_{rj})$.
We sometimes write $\sigma_{ij}^A$ for $\sigma_{ij}$ to stress the
role that $A$ plays. 
We refer to the vector of parities $\pi(\sigma_{ij})$ indexed by the
$k(k-1)$ pairs $(i,j)$ as the \textit{$\sigma$-parity of~$A$}.
\end{definition}

In \sref{ss:equivalencetausigma}, we establish an equivalence between
$\tau$-parity and $\sigma$-parity. Here, we consider some basic
properties of $\sigma$-parity. First, we observe the effect of
interchanging the indices.

\begin{lemma}\label{l:sigmaijvssigmaji}
Given an $\OA(k,n)$ and distinct integers $i,j \in [1,k]$, 
\begin{equation} \label{e:sigmacommutativity}
 \pi(\sigma_{ji}) = \pi(\sigma_{ij}) + {n \choose 2}.
\end{equation}
\end{lemma}

\begin{proof}
Let $A=(a_{r\ell})$ be an $\OA(k,n)$ on alphabet $\Lambda$, where 
$r\in \Lambda^2$ and $\ell \in [1,k]$.  Since $A$ is an $\OA$, there is
a bijection $R : \Lambda^2 \rightarrow\Lambda^2$ for which 
$R(u,v) = r$ where $u=a_{ri}$ and $v=a_{rj}$. Then 
$\sigma_{ji} = \sigma_{ij} 
\prod_{\{u,v\}} \big(R(u,v), R(v,u)\big)$, 
where $\big(R(u,v), R(v,u)\big)$ is a transposition
and the product is over all unordered pairs of distinct elements of $\Lambda$. 
The claim follows.
\end{proof}

Note that \eref{e:sigmacommutativity} relies on an essential
property of orthogonal arrays; namely, that every ordered pair of
distinct symbols occurs exactly once in every pair of
columns of an $\OA$. We refer to this equation to derive further
properties of both $\sigma$- and $\tau$-parities.

Next, we consider the analogue of \lref{l:tauparconj} for
$\sigma$-parity; that is, we consider the $\sigma$-parities of
isomorphic orthogonal arrays.

\begin{lemma}\label{l:sigmaparconj}
 Let $A_1$ and $A_2$ be two orthogonal arrays, $\OA(k,n)$, on an
 alphabet $\Lambda$.  Index the columns of $A_1$ and $A_2$ by $[1,k]$. 
 Let $i,j \in [1,k]$ be two distinct integers.
\begin{itemize}
\item[(i)] If $A_2$ is obtained by permuting the rows of $A_1$ by
  $\gamma$, then $\pi(\sigma^{A_2}_{ij})=\pi(\sigma^{A_1}_{ij}) + \pi(\gamma)$.
\item[(ii)] If $A_2$ is obtained by permuting the columns of $A_1$ by
  $\gamma$, then $\pi(\sigma^{A_2}_{\gamma(i)\gamma(j)})=\pi(\sigma^{A_1}_{ij})$.
\item[(iii)] Let $\gamma\in\sym_\Lambda$ and $i\in[1,k]$.  If $A_2$ is
  obtained by applying $\gamma$ to every entry in column $i$ of $A_1$,
  then $\pi(\sigma^{A_2}_{ij})= \pi(\sigma^{A_1}_{ij}) + n \cdot
  \pi(\gamma)$, and if $i' \neq i$ then
  $\pi(\sigma^{A_2}_{i'j})= \pi(\sigma^{A_1}_{i'j})$.
\end{itemize}
\end{lemma}

\begin{proof}
If $\gamma$ is a permutation of $\Lambda^2$ which acts on the rows of
$A_1$, then $\sigma^{A_2}_{ij} = \sigma^{A_1}_{ij} \gamma$, which implies the
first statement. Similarly, if $\gamma$ is a permutation of $[1,k]$
which acts on the columns of $A_2$, then 
$\sigma^{A_1}_{ij}=\sigma^{A_2}_{\gamma(i)\gamma(j)}$. 
Now assume that $\gamma$ is a permutation of
$\Lambda$ that is applied to every entry in column $i$ of
$A_1=(a_{r\ell})$.  Write
$\gamma$ as a product of transpositions: 
$\gamma = (v_1,v'_1)(v_2,v'_2) \cdots (v_m, v'_m)$ 
where $m\ge0$ and $v_l \neq v'_l$ for $l \in [1,m]$. 
Let $R$ be the bijection from the proof of \lref{l:sigmaijvssigmaji}. Then
\[
\sigma^{A_2}_{ij}  
= \sigma^{A_1}_{i j}\prod_{l \in [1,m]} \prod_{s \in \Lambda} 
\big( R(v_l, s), R(v'_l, s) \big).
\]
Therefore, 
$\pi(\sigma^{A_2}_{ij})  = \pi(\sigma^{A_1}_{i j}) + mn   = 
\pi(\sigma^{A_1}_{ij}) + n \cdot \pi(\gamma)$ in $\Z_2$.  
The claim is obvious for $i' \neq i$.
\end{proof}




\subsection{Equivalence between $\tau$-parity and $\sigma$-parity}
\label{ss:equivalencetausigma}

In this subsection, we show that the $\sigma$-parity of an
orthogonal array determines its $\tau$-parity. Also, the converse
statement is true up to complementation.

Let $A=(a_{r\ell})$ be an $\OA(k,n)$ and fix distinct $c,i,j\in[1,k]$.  
Consider a
row $r$ of $A$ and let $x=a_{rc}$, $y=a_{ri}$ and $z=a_{rj}$.  The
permutation $\sigma_{cj}\sigma_{ci}^{-1}$ maps $(x,y)$ to $(x,z)$.
Thus the contribution to $\pi(\sigma_{cj}\sigma_{ci}^{-1})$ from the
rows in which the symbol $x$ occurs in column $c$ is precisely
$\pi(\rho_x(c,i,j))$. Summing over $x$, we find that
\begin{equation}\label{e:taufromsigmas}
\tau^c_{ij} = \pi(\sigma_{cj}\sigma_{ci}^{-1}) = \pi(\sigma_{ci}\sigma_{cj}).
\end{equation}
This demonstrates that $\sigma$-parity determines $\tau$-parity.

Note that \lref{l:sigmaparconj}(i) implies that if two rows of an
$\OA$ are interchanged, then every $\pi(\sigma_{ij})$ changes value. 
We call this {\it $\sigma$-complementation}.
However, interchanging two rows of an orthogonal array
$A$ does not change $\M(A)$, hence one would not expect the
parity to change.   Also, permuting the rows of an $\OA$
does not affect $\tau$-parity (see \lref{l:tauparconj}(i)), which
means there is no hope of recovering $\sigma$-parity from $\tau$-parity.
However, if some kind of standardisation is imposed to
choose between a $\sigma$-parity and its $\sigma$-complement, then
this standardised choice may be recovered from the $\tau$-parity.
We will use two forms of standardisation in this paper. The simplest
is just to insist that $\pi(\sigma_{12})=0$.
Under this convention, it is easy to recover the $\sigma$-parity
given the $\tau$-parity of an orthogonal array, as follows:
\begin{align}
\pi(\sigma_{12}) &= 0, \label{e:convent} \\
\pi(\sigma_{1j}) &= \tau_{2j}^1 && \text{for } j \geq 3, \label{e:sig1j} \\
\pi(\sigma_{2j}) &= \tau_{1j}^2 + {n \choose 2} && \text{for } j \geq 3, 
     \label{e:sig2j}\\
\pi(\sigma_{ij}) &=  \tau_{2i}^1 + \tau_{1j}^i + {n \choose 2} && 
\text{for } 3 \leq i < j. \label{e:sigij}
\end{align}
As usual, all parity equations are in $\Z_2$.
The first of these four equations, \eref{e:convent}, is our
standardisation. By~\eref{e:taufromsigmas}, 
$\tau_{2j}^1 = \pi(\sigma_{12} \sigma_{1j}) = \pi(\sigma_{12}) +
\pi(\sigma_{1j})$ when $j \geq 3$, giving \eref{e:sig1j}. Note that
$\tau_{2j}^1 + \tau_{1j}^i 
= \pi(\sigma_{12}\sigma_{1j}\sigma_{i1}\sigma_{ij})$.
When $i=2$ and $j \geq 3$, using \eref{e:convent} and \eref{e:sig1j} 
as well as the property~\eqref{e:sigmacommutativity}, we have 
\begin{align*}
\tau_{2j}^1 + \tau_{1j}^2 
&= \pi(\sigma_{1j}) + \pi(\sigma_{21}) + \pi(\sigma_{2j})\\
&= \tau_{2j}^1 + {n \choose 2} + \pi(\sigma_{2j}),
\end{align*}
which gives \eref{e:sig2j}. Similarly, if $i \geq 3$, we have 
\begin{align*}
\tau_{2j}^1 + \tau_{1j}^i 
&= \pi(\sigma_{1j}) + \pi(\sigma_{i1}) + \pi(\sigma_{ij}) \\
& = \tau_{2j}^1 + \tau_{2i}^1 + {n \choose 2} + \pi(\sigma_{ij}),
\end{align*}
which gives \eref{e:sigij}.

In several arguments towards the end of \sref{s:ensemble} it will be
convenient to use a form of standardisation other than
\eref{e:convent}. However, any method which decides whether to take a
$\sigma$-parity or its $\sigma$-complement will allow
$\sigma$-parity to be recovered from $\tau$-parity.  We can simply
find the answer determined by \eref{e:convent}--\eref{e:sigij} and 
then take the $\sigma$-complement or not, as required.

We can use the relationship between $\tau$-parity and $\sigma$-parity 
to generalise \eref{e:fundLSpar}, and give a simple proof.

\begin{lemma}\label{l:LStriangleparity}
Suppose $c_1, c_2, \dots, c_l \in [1,k]$ are distinct integers, where
$l\geq 3$.  For any $\OA(k,n)$ we have, in $\Z_2$,
\[ 
\tau^{c_1}_{c_lc_2} + \tau^{c_2}_{c_1c_3} + \cdots + 
\tau^{c_i}_{c_{i-1}c_{i+1}} + \cdots + \tau^{c_l}_{c_{l-1}c_1}  
= l {n \choose 2}. 
\]
\end{lemma}

\begin{proof}
By~\eref{e:taufromsigmas} and \eref{e:sigmacommutativity},
\[
\tau^{c_1}_{c_lc_2} + 
\cdots + \tau^{c_i}_{c_{i-1}c_{i+1}} + \cdots + \tau^{c_l}_{c_{l-1}c_1} 
= \pi(\sigma_{c_1c_l}\sigma_{c_lc_1}) 
+\sum_{i\in[1,l-1]} \pi(\sigma_{c_i c_{i+1}} \sigma_{c_{i+1}c_i} ) 
= l {n \choose 2}. \qedhere 
\]
\end{proof}

Applying \lref{l:LStriangleparity} in the case where $l = 3$, we find
that for an $\OA(k,n)$ and three distinct integers $c,i,j \in [1,k]$
we have
\begin{equation}\label{e:Triple}
  \tau^c_{ij} + \tau^i_{cj} + \tau^j_{ci} = {n \choose 2}. 
\end{equation}
Of course, this is essentially a restatement of \eref{e:fundLSpar}.


\section{Graphs that model parities}\label{s:graphs}

We next describe graph theoretic interpretations for both our notions
of parity ($\tau$-parity and $\sigma$-parity). Our graphs will not
include loops or multiple edges, but will sometimes have directed
edges.  Given a graph or digraph $G$, we use $V(G)$ for the vertex set
of $G$ and $E(G)$ for the set of (possibly directed) edges.  For an
undirected graph $G$ we use $N_G(v) \subseteq V(G)$ to denote the
neighbourhood of $v$ in $G$, the set of vertices in $G$ which are
adjacent to a given vertex $v$.

The \textit{complement} of a undirected graph $G$, denoted
$\overline{G}$, is the graph obtained by replacing edges of $G$ with
non-edges and vice versa.  By \textit{switching} an undirected $G$ at
$v\in V(G)$ we obtain the graph, denoted $G^v$, which is equal to $G$
except that the neighbourhood of $v$ in $G^v$ is the complement of the
neighbourhood of $v$ in $G$. In other words, $N_{G^v}(v)=N_{\overline{G}}(v)$.

The \textit{reverse} of a digraph $G$, also denoted $\overline{G}$, is
the digraph obtained from $G$ by reversing the direction of every
edge.  By \textit{switching} a directed $G$ at $v\in V(G)$, we obtain
the digraph, denoted $G^v$, which is equal to $G$ except that the
direction of each edge incident with $v$ is reversed.  Note that we
use the same notation $\overline{G}$ and $G^v$, where the meaning is
determined by context, depending on whether $G$ is a undirected graph
or a digraph.

For a fixed initial digraph, the digraph obtained 
by applying a finite sequence of switchings and
reversals depends only on the parity of the number of switchings taken
at each vertex, and the parity of the number of reversals taken. It does
not depend on the order in which these operations are applied (see
\cite{MS13} for further details).  It is easy to check that analogous
properties hold in the context of undirected graphs as well.


\subsection{Graphs related to $\tau$-parity}

\begin{definition} 
Let $A$ be an $\OA(k,n)$.  For each $c \in [1,k]$, we define an
undirected graph $G_c$ with $V(G_c)=[1,k]$, where vertex $c$ is
isolated and $\{i,j\} \in E(G_c)$ if and only if $\tau_{ij}^c = 1$ for
$i,j\ne c$.  We call the graphs $G_1, \dots, G_k$ the
\textit{$\tau$-graphs} of $A$.
\end{definition}

We find it convenient to consider empty graphs to be complete
bipartite graphs (with one side of the graph having cardinality zero).

\begin{lemma}\label{l:taugraphbipartite}
Let $G_1, \dots, G_k$ be the $\tau$-graphs of an $\OA(k,n)$. Then, for
each $c \in [1,k]$, graph $G_c$ is the disjoint union of an isolated
vertex and a complete bipartite graph on $k-1$ vertices.
\end{lemma}

\begin{proof}
For some $c \in [1,k]$, let $G$ be $G_c$ with the isolated vertex $c$ removed.
If $G$ has no edges, then $G$ is 
$K_{0,k-1}$, so assume that $G$ has at least one edge. 
Let $i \in V(G)$ be such that $V_1=N_{G}(i) \neq \emptyset$. 
By~\eqref{e:Fixedcol}, there is an even number of edges 
between every triple of vertices in $G$, hence $V_1$  is an independent set.  
Now consider $V_2 = V(G) \backslash V_1$.  
For every $j \in V_2\setminus \{i\}$ and $\ell \in V_1$, we know that
$\{j,\ell\} \in E(G)$ because $j \not\in V_1=N_{G}(i)$ and there must be
an even number of edges induced by the vertices $\{i,j,\ell\}$. 
Finally, $V_2=N_G(\ell)$ for $\ell\in V_1$, 
so it is an independent set. Thus $G$ is a
complete bipartite graph with partite sets $V_1$ and $V_2$.
\end{proof}

We next state the analogue of \lref{l:tauparconj} for $\tau$-graphs.

\begin{lemma}\label{l:taugraphs} 
Let $A_1$ and $A_2$ be $\OA(k,n)$ and let $G_1, \dots, G_k$ be the
$\tau$-graphs of $A_1$ and $H_1, \dots, H_k$ be the $\tau$-graphs of $A_2$.
\begin{itemize}
\item[(i)] If $A_2$ is obtained by permuting the rows of $A_1$, then
$G_c$ is equal to $H_c$ for each $c \in [1,k]$.

\item[(ii)] If $A_2$ is obtained by permuting columns of $A_1$ by $\gamma$, 
then $\gamma$ is an isomorphism which maps $G_c$ to $H_{\gamma(c)}$  for all
$c \in [1,k]$.

\item[(iii)] Let $\gamma\in\sym_\Lambda$ and $c\in[1,k]$ and suppose
  $A_2$ is obtained by applying $\gamma$ to every entry in column $c$
  of $A_1$.  If $n$ is even or $\pi(\gamma)=0$, then $H_i=G_i$ for all
  $i \in [1,k]$. If $n$ is odd and $\pi(\gamma)=1$ then $H_c=G_c$ and,
  for $i\in[1,k]\setminus\{c\}$, we obtain $H_i$ by switching
  $G_i\backslash \{i\}$ at vertex $c$, with vertex $i$ remaining
  isolated.
\end{itemize}
\end{lemma}

We remark that the only operation that may produce isomorphic $\OA$s
with non-isomorphic $\tau$-graphs is an odd permutation of the symbols
in a column when the alphabet size is odd. This results in switchings
at the vertex that corresponds to the column being permuted.

The $\tau$-graphs have the same vertex set, and by
\lref{l:taugraphbipartite}, they are a disjoint union of an isolated
vertex and a complete bipartite graph. Next we study what happens if
we superimpose them, modulo 2. Define the {\em stack} corresponding to
an $\OA(k,n)$ to be the undirected graph with vertices $[1,k]$ and an
edge $\{i,j\}$ if and only if $\sum_c\tau^c_{ij}\= 1\mod2$ (in other
words, the edge $\{i,j\}$ is present in an odd number of the
$\tau$-graphs for the $\OA$).  The following lemma is an interesting
observation about the stack corresponding to an $\OA(k,n)$.


\begin{theorem}\label{t:stackedgraph} 
Let $G$ be the stack for an $\OA(k,n)$. Then 
\begin{itemize}
\item[(i)] $G$ is a complete bipartite graph if $n \equiv 0,1 \mod{4}$, and
\item[(ii)] $G$ is a vertex disjoint union of at most two complete
  graphs if $n \equiv 2,3 \mod{4}$.
\end{itemize}
\end{theorem}

\begin{proof}
We consider the number of edges between any three distinct vertices 
$i,j,\ell \in [1,k]$. Working in $\Z_2$, we have,
\begin{align*}
 \sum\limits_{c \neq i,j} 
\tau^c_{ij} \, + \sum\limits_{ c \neq j,\ell} 
\tau^c_{j\ell} \, + \sum\limits_{c \neq \ell,i} 
\tau^c_{\ell i}  
&=  \pi \bigg( \prod_{c \neq i,j} 
\sigma_{ci} \sigma_{cj}  \,  \prod_{ c \neq j,\ell} 
\sigma_{cj} \sigma_{c\ell}  \,  \prod_{c \neq 
\ell,i} \sigma_{c\ell} \sigma_{ci}\bigg) && \mbox{by~\eref{e:taufromsigmas}} \\ 
 &= \pi \big( \sigma_{\ell i} \sigma_{\ell j} \, \sigma_{ij} \sigma_{i \ell} 
\, \sigma_{j\ell} \sigma_{j i}\big) && \\
 &= 3 {n \choose 2}. &&  \mbox{by~\eref{e:sigmacommutativity}.} 
\end{align*}
Therefore, when $n \equiv 0,1 \mod{4}$, there is an even number of edges 
between any three distinct vertices in $G$. Analogous to the proof of 
\lref{l:taugraphbipartite}, this implies that $G$ must be a complete 
bipartite graph. 

If $n \equiv 2,3 \mod{4}$, there is an odd number of edges between any
three distinct vertices in $G$.  In particular, there are no induced
paths of two edges. Since there cannot be two vertices at distance $2$
from each other, it follows that each component of $G$ is a complete
graph.  There are at most two components, since otherwise there would
be three vertices inducing a graph with no edges, but any three
vertices must induce an odd number of edges.
\end{proof}

In \cyref{c:stackPP} we give a stronger restriction on the stack
corresponding to an $\OA(n+1,n)$.


\subsection{Graphs related to $\sigma$-parity}

\begin{definition} 
Let $A$ be an $\OA(k,n)$.  The {\it $\sigma$-matrix} of $A$ is the
$k\times k$ matrix $M = (m_{ij})$ where
\begin{displaymath}
m_{ij} = \left\{ \begin{array}{ll}
1 & \textrm{if } i \ne j \textrm{ and } \pi(\sigma_{ij}) = 1, \\
0 & \textrm{otherwise.}
\end{array} \right.
\end{displaymath}  
The {\it $\sigma$-graph} of $A$, denoted by $\G(A)$, is a 
graph or digraph on $k$ vertices with adjacency matrix $M$. 
If $M$ is symmetric, then we interpret
$\G(A)$ to be an undirected graph. Otherwise, we interpret
$\G(A)$ to be a digraph.
\end{definition}

\begin{lemma}\label{l:Sym}
Let $A$ be an $\OA(k,n)$.  If $n\equiv 0,1 \mod{4}$, then the
$\sigma$-matrix is symmetric and the $\sigma$-graph is an
undirected graph.  If $n \equiv 2,3 \mod{4}$, then the $\sigma$-graph
is a tournament.
\end{lemma}

\begin{proof}
By~\eqref{e:sigmacommutativity}, we have 
$\pi(\sigma_{ji}) = \pi(\sigma_{ij}) + {n \choose 2}$ for distinct
$i,j \in [1,k]$.  If $n \equiv 0,1 \mod{4}$, then $\sigma_{ij}$ and
$\sigma_{ji}$ have the same parity, so the $\sigma$-matrix of $A$ is
symmetric.  If $n \equiv 2,3 \mod{4}$, then $\sigma_{ij}$ and
$\sigma_{ji}$ have opposite parity, so exactly one of the directed
edges $(i,j)$ and $(j,i)$ is in $\G(A)$.  Hence $\G(A)$ is a tournament.
\end{proof}

We next state the analogue of \lref{l:sigmaparconj} for
$\sigma$-graphs.  Glynn and Byatt \cite[Lem.\,2.4]{GB12} showed that
$\sigma$-graphs of orthogonal arrays $\OA(k,n)$ for $n$ even are
invariant (up to complementation) for isomorphic orthogonal arrays.
Here we consider $n$ odd as well. The proof is essentially the
same as for the even case, so we omit it.

\begin{lemma} \label{l:sigmagraphs}
Let $A_1$ and $A_2$ be $\OA(k,n)$ and let $\G_1=\G(A_1)$ and $\G_2=\G(A_2)$. 
\begin{itemize}
\item[(i)] If $A_2$ is obtained by permuting the rows of $A_1$ by
  $\gamma$, then $\G_2$ is either $\G_1$ if $\gamma$ is an even
  permutation, or $\overline{\G_1}$ if $\gamma$ is an odd permutation.

\item[(ii)] If $A_2$ is obtained by permuting the columns of $A_1$ by
  $\gamma$, then $\gamma$ is a graph isomorphism which maps $\G_1$ to $\G_2$.

\item[(iii)] Let $\gamma\in\sym_\Lambda$ and $i \in[1,k]$ and suppose
  $A_2$ is obtained by applying $\gamma$ to every entry in column $i$
  of $A_1$.  Then $\G_2 = \G_1$ if $\pi(\gamma)=0$ or
  $n$ is even, and $\G_2 = \G_1^{i}$ if $\pi(\gamma)=1$ 
  and $n$ is odd.
\end{itemize}
\end{lemma}

We see in \lref{l:sigmagraphs} that taking a conjugate of a
set of MOLS only affects the order of columns of the orthogonal
arrays, and hence yields an isomorphic $\sigma$-graph. On the other
hand, the $\sigma$-graphs of isotopic sets of MOLS are not necessarily
isomorphic if $n$ is odd, because isotopisms can cause switchings
at some vertices.

\section{How many parities are there?}\label{s:numpar}

If $A$ is an $\OA(k,n)$ then there is a corresponding $\tau$-parity
$\tau(A)$ of dimension $k(k-1)(k-2)$ which stores the
parities $\tau^c_{ij}$, indexed by the triples $(c,i,j)$.  In this
section we are interested in the number of different $\tau$-parities
that are possible. One way to measure this is to consider $\dim(k,n)$,
which we define as $\log_2$ of the number of different $\tau$-parities
achieved by orthogonal arrays $\OA(k,n)$. We have already determined
a number of relationships between components of the $\tau$-parity,
including \eqref{e:tauidxcommutativity}, \eqref{e:Fixedcol} and
\eqref{e:Triple}.  Studying $\dim(k,n)$ is one way to determine
whether there are other relationships waiting to be discovered. We
will be able to resolve this question in some cases, but it will
remain open in general.  Another way to view $\dim(k,n)$ is that it is
the information content (in number of bits) of $\tau(A)$ for an
$\OA(k,n)$.

For a given value of $(k,n)$, we call any $\Z_2$ vector satisfying
\eqref{e:tauidxcommutativity}, \eqref{e:Fixedcol} and \eqref{e:Triple}
a {\it plausible $\tau$-parity}.  We call it an {\it actual
  $\tau$-parity} if there is an $\OA(k,n)$ that achieves it. The set
of plausible $\tau$-parities depends only on the value of $n\mod4$ and
on $k$.


\subsection{Switching classes}\label{ss:swclass}

Once we have an actual $\tau$-parity, there will be many other actual
$\tau$-parities that we can establish by taking isomorphic orthogonal
arrays.  For a plausible $\tau$-parity $p$, the \textit{switching
  class} of $p$ is the set of all plausible $\tau$-parities obtained
from $p$ by some sequence of the following operations. These
operations are analogues of (ii) and (iii) from \lref{l:tauparconj},
and will map any actual $\tau$-parity to another actual $\tau$-parity.

\textbf{Permuting Operation:} \
For a permutation $\gamma \in\sym_k$, a new $\tau$-parity is
obtained by {\it permuting by $\gamma$} as follows:
we move each $\tau_{ij}^c$ from the coordinate indexed by $(c,i,j)$ to
the coordinate indexed by $\big(\gamma(c),\gamma(i),\gamma(j)\big)$.

\textbf{Swapping Operation:} \ 
If $n$ is odd then a new $\tau$-parity is obtained by {\em swapping at
  a subset $C \subseteq [1,k]$}, which means for each $\tau_{ij}^c$, we
change its value if and only if $|\{i,j\} \cap C| = 1$.

We stress that the swapping operation is only available if $n$ is 
odd.  

\begin{theorem}\label{t:switchsizes}
Let $k \geq 3$ and $n \geq 2$.  Then each switching class for
parameters $(k,n)$ has size which divides $k!$ if $n$ is even and
$k!\,2^{k-1}$ if $n$ is odd.
\end{theorem}

\begin{proof}
First, suppose that $n$ is even.  The group $\sym_k$ acting on the
labels $c,i,j$, induces an action on the $\tau$-parities.  The
switching classes are the orbits of this action. Hence
the result follows.

The situation for odd $n$ is similar, except that we have the option
to swap at any subset $C\subseteq[1,k]$. However, swapping at $C$ has
the same effect as swapping at the complement of $C$.  We can remove
this duplication by marking one index in $[1,k]$, and never swapping
at the marked index (the mark gets moved when we permute indices). In
this way we find an action of the group $\sym_2^{k-1} \rtimes \sym_k$
on $\tau$-parities.  This group has order $k!\,2^{k-1}$ and its orbits
are the switching classes.
\end{proof}

We remark that $\sym_2^{k-1} \rtimes \sym_k$ is the automorphism group of 
the folded $k$-cube (at least when $k>4$), see \cite[pp.\,264--265]{BCN89}.

For $k\in[3,8]$ and for each possible value of $n\mod4$, we calculated
the number of switching classes and the size of each switching class
for every plausible $\tau$-parity.  The results are given in
\Tref{T:switchclass}. Instances that achieve the bounds in
\tref{t:switchsizes} are shown in {\bf bold}.  Note that for odd $n$
the bounds are not achieved until $k=8$.

\begin{table}[h!]
\centering
\footnotesize
\begin{tabular}{|c|c|l|c|l|}
\hline
$k$ & \multicolumn{2}{c|}{$n \equiv 0 \mod{4}$} &  
      \multicolumn{2}{c|}{$n \equiv 2 \mod{4}$} \\[0.5ex]
 & \# & sizes & \# & sizes   \\
\hline
3 &2& $ 1 , 3 $  &2&  $ 1 , 3 $  \\[1ex]
4 &6& $ 1 , 3 , 4 , 6 , 12 $ &3& $ 8 , 12 $ \\[1ex] 
5 &18& $ 1 , 5 , 6 , 10 , 15 , 20 , 30 , 60 $  &10&  $ 12 , 20 , 40 , 60 , {\bf 120} $\\[1ex]
6 &78& $ 1 , 6 , 10 , 15 , 20 , 30 , 45 , 60 , 72 , 90 , 120 , 180 , 360 , {\bf 720}
$ &34& $ 40 , 120 , 144 , 240 , 360 , {\bf 720} $ \\[1ex]
7 &522& $ 1 , 7 , 21 , 35 , 42 , 70 , 105 , 140 , 210 , 252 , 315 , 360 , 420 , $ 
 &272& $ 120 , 280 , 360 , 504 , 560 , 840 ,  $ \\
 && $504 , 630 , 840 , 1260 , 2520 , {\bf 5040}$ 
 && $1008 , 1680 , 2520 , {\bf 5040}$\\[0.5ex]
8&6178& $1,8,28,35,56,70,105,168,210,280,315,336,$&3528&$1920,2240,2688,4480,$\\
 && $420,560,630,672,840,1120,1260,1680,2016,2520,$&&$5760,6720,8064,13440,$\\
 && $2880,3360,4032,5040,6720,10080,20160,{\bf40320}$&&$20160,{\bf40320}$\\
\hline
\end{tabular}

\medskip

\footnotesize
\begin{tabular}{|c|c|l|c|l|}
\hline
$k$ & \multicolumn{2}{c|}{$n \equiv 1 \mod{4}$} &  
      \multicolumn{2}{c|}{$n \equiv 3 \mod{4}$} \\[0.5ex]
 & \# & sizes & \# & sizes   \\
\hline
3 &1&  $ 4 $ &1& $ 4 $ \\[1ex]
4 &2&  $ 8 , 24 $ &2& $ 8 , 24 $\\[1ex]
5 &4&  $ 16 , 96 , 160 , 240 $ &2& $ 192 , 320 $ \\[1ex]
6 &10&  $ 32 , 192 , 320 , 480 , 1440 , 1920 , 2880 , 5760 $
 &6& $ 640 , 1920 , 2304 , 3840 , 5760 $ \\[1ex]
7 &27&  $ 64 , 1344 , 2240 , 4480 , 6720 , 13440 , 16128 , 20160 ,$ &12& $ 7680 , 
17920 , 23040 , $ \\
 && $ 23040 , 26880, 40320 , 53760 , 80640 , 161280$ && $ 32256 , 53760 , 
161280$\\[1ex]
8 &131& $ 128, 3584, 4480, 7168, 13440, 21504, 26880, 35840,$ &69& $15360, 143360, 172032, 215040,$ \\
&& $ 40320, 53760, 71680, 86016, 107520, 161280, 215040, $ && $286720, 322560, 368640, 430080, $
\\
&& $258048, 322560, 368640, 430080, 645120, 860160,$ && $516096, 645120, 860160, 1290240,$ \\
&& $ 1290240, 2580480, {\bf 5160960}$ && $1720320, 2580480, {\bf5160960}$ \\[0.5ex]
\hline
\end{tabular}
\caption{\label{T:switchclass}Number and sizes of switching classes.}
\end{table}

Each plausible $\tau$-parity corresponds to a $\sigma$-graph up to
complementation, and vice versa. For even $n$, a switching class is
defined by the permuting operation, which simply applies an
isomorphism to the $\sigma$-graph. It follows that the number of
switching classes for $(k,n)$ when $n\= 0\mod4$ is the number of
complementary pairs of graphs on $k$ nodes (sequence A007869 in
\cite{OEIS}). Similarly, the number of switching classes for $(k,n)$
when $n\= 2\mod4$ is the number of complementary pairs of tournaments
on $k$ nodes (sequence A059735 in \cite{OEIS}).


Up to isotopism there are 19 complete sets of MOLS of order 9, falling
into 7 isomorphism classes of $\OA(10,9)$.  For each complete set $M$
of MOLS, we constructed $\A(M)$, the corresponding $\OA(10,9)$. We then
determined its $\tau$-parity, and the switching class it belongs to.
The results are shown in \Tref{TableIsom}. Sets of MOLS are numbered
using the numbering in \cite{handbook}. Set number 1 corresponds to
the Desarguesian plane, sets 2 to 6 correspond to the Hall plane, sets
7 to 11 correspond to the dual Hall plane, and sets 12 to 19
correspond to the Hughes plane.
It turns out that only 4 different switching classes are represented
by the 19 complete sets. These 4 switching classes do {\em not}
correspond to the 4 projective planes. In fact one of the switching
classes contain OAs from three different projective planes, and
another contains OAs from two different projective planes, as indicated
in the last column of \Tref{TableIsom}.  Clearly,
not all OAs from a given plane end up in the same switching class.

\begin{table}[h!]
\begin{center}
\begin{tabular}{ |c|c|l|c| } 
 \hline
 Switching Class & Size of class
& Sets of $8$ MOLS(9) in the class 
& \# Projective planes\\
 \hline
 Class 1 & 1290240 & 1 & 1 \\ 
 Class 2 & 483840 & 2,3,10,11 & 2\\ 
 Class 3 & 512 & 4,5,6,7,8,9,17,18,19 & 3 \\
 Class 4 & 1075200  & 12,13,14,15,16 & 1\\
 \hline
\end{tabular}
\caption{\label{TableIsom}Switching classes of complete sets of MOLS(9)}
\end{center}
\end{table}

\Tref{TableIsom} also shows the size of each switching class and here
it is slightly surprising that the most symmetric plane (the
Desarguesian one) produces OAs in the largest switching class. In fact
the smallest switching class contains the zero vector.  Each of the
non-Desarguesian planes have OAs which achieve this zero vector, but
the Desarguesian plane does not.

\subsection{Achieving all plausible $\tau$-parities}\label{ss:plausisactual}

We next consider the question of how many plausible $\tau$-parities there
are. After that we will consider how many of them are actual $\tau$-parities.

\begin{theorem}\label{t:sigmaplausible}
There are $2^{{k\choose2}-1}$ plausible $\tau$-parities for an
$\OA(k,n)$, because each standardised $\sigma$-parity corresponds to a
different plausible $\tau$-parity.
\end{theorem}

\begin{proof}
We choose the standardisation \eref{e:convent}.  By
\sref{ss:equivalencetausigma} we know that each standardised
$\sigma$-parity determines a different $\tau$-parity, so we consider
the number of options for the $\sigma$-matrix.  Our choice for
$\pi(\sigma_{12})$ is fixed but we can freely pick $\pi(\sigma_{ij})$
for all other $1\le i<j\le k$. This gives us $2^{{k\choose2}-1}$
options, but we have to check that all of them produce plausible
$\tau$-parities.  Suppose we have settled on values for
$\pi(\sigma_{ij})$ for $1\le i<j\le k$. We then determine
$\pi(\sigma_{ij})$ for $1\le j<i\le k$ from
\eref{e:sigmacommutativity}, and $\tau^c_{ij}$ for all distinct
$c,i,j\in[1,n+1]$ from \eref{e:taufromsigmas}.

It is straightforward to check 
\eqref{e:tauidxcommutativity} and \eqref{e:Fixedcol} hold.  
Also, in $\Z_2$,
\begin{align*}
\tau^c_{ij} + \tau^i_{cj} + \tau^j_{ci} 
&= \pi(\sigma_{ci}) + \pi(\sigma_{cj}) + \pi(\sigma_{ic}) + \pi(\sigma_{ij}) 
+ \pi(\sigma_{jc}) + \pi(\sigma_{ji}) && \textrm{ by \eqref{e:taufromsigmas}}\\
&=  3 {n \choose 2}, && \textrm{ by \eqref{e:sigmacommutativity}} 
\end{align*}
and hence \eqref{e:Triple} is satisfied.
\end{proof}

\begin{corollary}\label{c:dim}
We have $\dim(k,n)\le { {k \choose 2} -1}$ for $k \geq 3$ and arbitrary $n$.
\end{corollary}

In \cyref{c:dimPP} we find a better upper bound on $\dim(k,n)$ in the
case when $k = n+1$.

We now know there are $2^{ {k \choose 2} -1}$ plausible $\tau$-parities.  
How many of these are actual $\tau$-parities?  In
\Tref{T:actualparvec}, we list some values of $(k,n)$ for which all
plausible $\tau$-parities are actual.  Examples
that justify each claim made in the table can be downloaded from
\cite{WWWW}.  This table is exhaustive for $n\le9$, where a complete
catalogue of MOLS is known \cite{EW16,WWWW}. For $n\ge10$ we used
ad-hoc methods to find our specimens, so no inference should be
made from an entry not appearing.

\begin{table}[h!]
\centering
\begin{tabular}{|c|c|c|c|c|}
\hline
$n \mod{4}$ & $k=3$ & $k=4$ & $k=5$\\
\hline
0 & 8  & 8   &  16   \\
1 & $5,9$  & 9   &  9        \\
2 & 6  & $10,14,18$	&     \\
3 & $3,7$  & 7   & $11,19,23,\dots$ \\
\hline
\end{tabular}
\caption{\label{T:actualparvec}Values of $(k,n)$ for which all
plausible $\tau$-parities are achieved.}
\end{table}

As \Tref{T:actualparvec} shows, we constructed explicit examples which
achieve all plausible $\tau$-parities for $k\in\{3,4,5\}$ and each
congruence class of $n$ mod 4 (except for $n \equiv 2$ mod {4} when
$k=5$). This exception is not a genuine one, since it will follow from
\cyref{c:actualNsufflarge} below that all plausible $\tau$-parities
are actual for all sufficiently large $n$ when $k=5$.

Note that $n=9$ is the smallest order which achieves all plausible
$\tau$-parities for $k=5$. Among the $\OA(6,9)$, there are 13312 actual
$\tau$-parities, and 16384 plausible $\tau$-parities.  The entry for
$n\equiv3\mod4$ and $k=5$ in \Tref{T:actualparvec} is justified by our
next result. In it we construct an infinite family of orthogonal arrays
which achieve all plausible $\tau$-parities when
$k=5$ (and hence also for smaller $k$). 

\begin{theorem}\label{t:InfiniteFamAllPar} 
For prime $n \geq 11$ satisfying $n \equiv 3 \mod{4}$ there exist
examples of $\OA(5,n)$ which achieve all $512$ plausible $\tau$-parities.
\end{theorem}

\begin{proof}
Suppose $n \geq 11$ is prime and $n \equiv 3 \mod{4}$.  Let $A$ be the
$\OA(n+1,n)$ corresponding to $n-1$ MOLS$(n)$, $\{ L_1, \dots, L_{n-1}\}$, 
where $L_\lambda[r,c] = \lambda r + c \mod{n}$ for
$\lambda \in [1,n-1]$.  First we determine $\tau_{ij}^c$ for every 
$c,i,j \in [1,\dots,k]$.

It is straightforward to check that $\tau_{ij}^1 = 0$ for every 
$i,j\in [2,k]$ since each of the permutations $\rho_a(1,i,j)$
is of the form $x \mapsto x+d\mod n$ for a constant $d$. 
Observe that the parity of a permutation $x \mapsto bx+d$ is the
same as the parity of $x \mapsto bx$. 
If $b$ has (multiplicative) order $q$ in $\Z_n$ then the permutation
$x\mapsto bx$ has cycle structure consisting of $({n-1})/{q}$ cycles
of length $q$.  Hence it has even parity if and only if
$({n-1})(q-1)/q$ is even, which occurs if and only if $q$ is odd,
given that $n \equiv 3 \mod{4}$. Now, $\rho_a(c,1,j)$ is the
permutation $x\mapsto (j-c)x+a$ for $j\ge2$. For $j>i \geq 2$, we see that 
$\rho_a(c,i,j)$
is the permutation $(i-c)x + a \mapsto (j-c)x+a$, which has the same
parity as the permutation $x\mapsto(i-c)^{-1}(j-c)x$. It follows that
in $\Z_2$,
\begin{align*}
& \tau_{ij}^1 = 0, \\
& \tau_{1j}^c = q-1, \text{ for }j\ge2,\text{ where } q \text{ is the order of $j-c$, and}  
\\
& \tau_{ij}^c = q-1, \text{ for }j>i\ge2, \text{ where } q \text{ is the order of } 
(i-c)^{-1}(j-c)
\end{align*}
Here order means multiplicative order, modulo $n$.

Let $a-1,a,a+1 \in \Z_n$ all be quadratic non-residues and define
$A_1=\A(L_a, L_{a^2},L_{a^3})$.  
Thus, columns $1,2,3,4,5$ of $A_1$ correspond to columns
$1,2, a+2, a^2+2, a^3+2$ of $A$.  The entire vector $\tau(A_1)$ 
can be reconstructed once we work out that
$$
[\tau_{ 2 3 }^ 1, \hspace{0.1cm}
 \tau_{ 2 4 }^ 1, \hspace{0.1cm}
 \tau_{ 2 5 }^ 1, \hspace{0.1cm}
 \tau_{ 1 2 }^ 3, \hspace{0.1cm}
 \tau_{ 1 2 }^ 4, \hspace{0.1cm}
 \tau_{ 1 3 }^ 4, \hspace{0.1cm}
 \tau_{ 1 2 }^ 5, \hspace{0.1cm}
 \tau_{ 1 3 }^ 5, \hspace{0.1cm}
 \tau_{ 1 4 }^ 5 \hspace{0.1cm}] 
= [0, 0, 0, 0, 1, 1, 0, 0, 0].$$
These values are not hard to derive. For example, $\tau_{13}^5$ for
$A_1$ corresponds to $\tau_{1j}^c$ for $A$ where $j = a+2$ and 
$c=a^3+2$.  Since $j-c = a-a^3 = (-a)(a+1)(a-1)$ is a quadratic residue,
it follows that $q$ is odd and hence $\tau_{13}^5=0$ for $A_1$.


Let $a-1,a,a+1 \in \Z_n$ be such that $a$ is a quadratic
non-residue and $a-1$ and $a+1$ are quadratic residues. Let 
$A_2=\A(L_a, L_{a^2}, L_{a^3})$, so
columns $1,2,3,4,5$ of $A_1$ correspond to columns $1,2, a+2, a^2+2,
a^3+2$ of $A$.  As before, we determine $\tau(A_2)$ by finding
$$
[\tau_{ 2 3 }^ 1, \hspace{0.1cm}
 \tau_{ 2 4 }^ 1, \hspace{0.1cm}
 \tau_{ 2 5 }^ 1, \hspace{0.1cm}
 \tau_{ 1 2 }^ 3, \hspace{0.1cm}
 \tau_{ 1 2 }^ 4, \hspace{0.1cm}
 \tau_{ 1 3 }^ 4, \hspace{0.1cm}
 \tau_{ 1 2 }^ 5, \hspace{0.1cm}
 \tau_{ 1 3 }^ 5, \hspace{0.1cm}
 \tau_{ 1 4 }^ 5 \hspace{0.1cm}] 
= [0, 0, 0, 0, 1, 0, 0, 0, 1].$$
Consulting our computation of the switching classes we find
$\tau(A_1)$ in the switching class 
of size 192, while $\tau(A_2)$ belongs to the switching class of size 320.

Jacobsthal \cite{Jac06} showed that for prime $n \equiv 3 \mod{4}$,
there exists a sequence of three consecutive elements which are
quadratic non-residues for prime $n > 7$ and there exists a sequence of
consecutive elements which are residue, non-residue, residue for prime
$n > 3$.  Thus, for prime $n \equiv 3 \mod{4}$, $n \geq 11$, there
are elements $a-1, a, a+1 \in \Z_n$ such that the above
constructions produce actual $\tau$-parities belonging to each of the two
switching classes.
\end{proof}

In principle, similar families can be constructed for other
cases. However, examination of \Tref{T:switchclass} reveals that
the case that we solved in \tref{t:InfiniteFamAllPar} involves fewer
switching classes (and hence less work) than would be required for
different values of $n\mod4$ or for larger values of $k$.

In the remainder of this section we demonstrate a different way to
achieve infinite families that achieve all plausible $\tau$-parities.
Specifically, we will show that if there are examples of $\OA(k,n)$
achieving all plausible $\tau$-parities for some $n$, then there are
examples of $\OA(k,N)$ achieving all plausible $\tau$-parities for all
sufficiently large $N$.  We first require some definitions.

An {\it incomplete Latin square} $L = (l_{ij})$ of order $n$ with a
hole of order $h$ is an $n \times n$ array on an $n$-set $\Lambda$
with a {\it hole} $H \subseteq \Lambda$ such that each cell $(i,j)$ is
empty if $\{i,j\} \subseteq H$ and contains exactly one symbol
otherwise, every row and every column of $L$ contains each symbol at
most once, and rows and columns indexed by $H$ do not contain symbols
in $H$.  A pair of incomplete Latin squares of order $n$ with common
hole of order $h$ are {\it orthogonal} if, when superimposed, all of the
$n^2 - h^2$ ordered pairs of symbols in $\Lambda^2\,\backslash H^2$
occur amongst the cells $(i,j)$ where $\{i,j\} \not\subseteq H$.  A
set of incomplete Latin squares of order $n$ with a common hole of
order $h$ which are pairwise orthogonal is called a set of incomplete
MOLS.

An {\it incomplete pairwise balanced design} on $n$ points with hole
size $h$ and blocks sizes $K$, denoted $IPBD((n,h),K)$, is a triple
$(V,H,\blocks)$ such that $V$ is a set of $n$ points, $H$ is a
subset of $V$ of size $h$ called the {\it hole}, and $\blocks$ is
a collection of subsets of $V$ where $|B| \in K$ for each $B \in
\blocks$ and every pair of points not both in $H$ occur together
in exactly one block.  Combining a result from \cite{CD11} 
(given as Construction 6.1.2 in \cite{vB15}) and 
\cite[Prop.~6.1.1]{vB15}, we have the following theorem.

\begin{theorem}\label{t:IMOLSconstruction}
If there exists an $IPBD((n,h),K)$ and for each $k \in K$ there exists
$t+1$ MOLS of order $k$, then there exist $t$ incomplete MOLS of order
$n$ with a common hole of order $h$.
\end{theorem}

The following proof is due to Peter Dukes.

\begin{theorem}\label{t:asymIMOLSexistence}
For any fixed $t$ and $h$ and for all sufficiently large $n$, there exist
$t$ incomplete MOLS of order $n$ with a common hole of order $h$.
\end{theorem} 

\begin{proof}
Let $K = \{2^a, 2^{a+1}, 3^b\}$ where $2^a>t+1$ and $3^b>t+1$.
Then, by McNeish's Theorem (see, e.g.~\cite[Thm~1.1.2]{vB15}), 
there exist $t+1$ MOLS of order $v$ for $v \in \{2^a,
2^{a+1}, 3^b \}$.  Also, by \cite[Thm~5.1.2]{vB15} and the
choice of $K$, there exists an $IPBD((n,h),K)$ for all sufficiently
large $n$.  The result follows by \tref{t:IMOLSconstruction}.
\end{proof}

We remark that Dukes and van Bommel \cite{DvB15} proved that there
exists a set of $t$ incomplete MOLS of order $n$ with a common hole of
order $h$ for all sufficiently large $n$ and $h$ satisfying $n \geq 8(t+1)^2h$.

\begin{theorem}\label{t:actualNsufflarge}
Let $N$ be sufficiently large relative to $n$.
The number of actual $\tau$-parities for $(k,N)$ is no less than
the number of actual $\tau$-parities for $(k,n)$.
\end{theorem}

\begin{proof}
Let $A_1$ and $A_2$ be two orthogonal arrays $\OA(k,n)$ with different
$\tau$-parities.  By \tref{t:asymIMOLSexistence}, for all sufficiently
large $N$, there exists an $\OA(k,N)$, say $A'_1$, that contains $A_1$
as a subarray.  Let $A'_2$ be the $\OA(k,N)$ obtained from $A'_1$ by
replacing the subarray $A_1$ by $A_2$.
Since $A_1$ and $A_2$ have different $\tau$-parities, there exist distinct
$c,i,j\in [1,k]$ such that $\tau^c_{ij}({A_1})\ne\tau^c_{ij}({A_2})$. 
This forces $\tau^c_{ij}({A'_1})\ne\tau^c_{ij}({A'_2})$, so
$A'_1$ and $A'_2$ have different $\tau$-parities. 
The result follows.
\end{proof}

\begin{corollary}\label{c:actualNsufflarge}
If all plausible $\tau$-parities are actual for $(k,n)$ then, for all
sufficiently large $N$, all plausible $\tau$-parities are actual for
$(k,N)$.
\end{corollary}

On the basis of the evidence gathered above, we propose:

\begin{conjecture}\label{cj:allactual}
For any fixed $k$ and all sufficiently large $N$, all plausible
$\tau$-parities are actual for $(k,N)$.
\end{conjecture}

When $k$ and $n$ are comparable in size we might expect there to be
plausible $\tau$-parities that are not actual for $(k,n)$. Indeed, we
will see in the next section that this is definitely the case when
$k=n+1$.

We close the section with an observation related to the comments at the
end of \sref{ss:tau-parities}. 
If $M=\M(A)$ where $A$ is an $\OA(k,n)$, then the
individual Latin squares in $M$ have at most $2(k-2)$ independent bits
of information in their row, column and symbol parities. For $k\ge4$,
this is strictly less than ${k\choose 2}-1$, which suggests, on the
basis of everything we have seen in this section, that a set of MOLS
has more information in its $\tau$-parity than there is in the
$\tau$-parity of its constituent Latin squares. Certainly, this is
true for pairs and triples of MOLS of large orders, by combining
\cyref{c:actualNsufflarge} and \Tref{T:actualparvec}.

\section{Projective planes}\label{s:PP}

In this section, we study the important special case of orthogonal
arrays which correspond to finite projective planes.  To begin,
we use the well-known fact that such orthogonal arrays correspond to
sharply $2$-transitive sets of permutations to derive a constraint on
the $\tau$-parity of the $\OA$.  A set $S$ of permutations of a set $X$ 
is said to be \textit{sharply
$2$-transitive} if for all $x,x',y,y' \in X$ where $x \ne x'$
and $y \ne y'$, there exists a unique $\gamma \in S$ such that
$\gamma(x) = y$ and $\gamma(x') = y'$.


\begin{theorem} 
  Let $A=(a_{r\ell})$ be an orthogonal array $\OA(n+1,n)$ on alphabet
  $\Lambda$, where $r \in \Lambda^2$ and $\ell \in [1,n+1]$. Let $i,j
  \in [1,n+1]$ be distinct integers. Then, in $\Z_2$,
\begin{equation} \label{e:PPcond}
  \sum_{c \in [1,n+1]\setminus\{i,j\} } \tau_{ij}^c = {n \choose 2}.
\end{equation}
\end{theorem}

\begin{proof}
Let $C_{ij} = [1,n+1]\setminus\{i,j\}$.
By definition, $\tau^c_{ij} = \sum_{x \in \Lambda} \pi(\rho_x(c,i,j))$
for $c\in C_{ij}$. We exploit the fact that 
$\{\rho_x(c,i,j)  :  x \in \Lambda, c \in C_{ij}\}$ 
is a sharply $2$-transitive set of permutations.  
Indeed, let $u, u', v, v' \in \Lambda$ be such that $u < u'$ and $v>v'$. 
Let $r, r' \in \Lambda^2$ be such that $a_{ri} = u$, $a_{rj}
= v$, $a_{r'i} = u'$, $a_{r'j} = v'$. 
No two rows of $A$ agree in more than one column, and overall 
there are $(n+1)n{n\choose 2}$ places in $A$ where two rows agree.
Since this matches the number of pairs of rows, 
any two rows must agree in exactly one column. 
Thus, there is a unique $c \in C_{ij}$ such that
$a_{rc}=a_{r'c}=x$, and $\rho_x(c,i,j)$ maps $u\mapsto v$
and $u'\mapsto v'$, producing an inversion. So in $\Z_2$,
\[
\sum_{c \in C_{ij} } \tau_{ij}^c = \sum_{c \in C_{ij} } \sum_{x \in 
\Lambda} \pi\big(\rho_x(c,i,j)\big)  = {n \choose 2}^2 = {n \choose 2},
\]
because each of the possible ${n \choose 2}{}^2$ inversions occurs
exactly once amongst the permutations in our sharply $2$-transitive set.
\end{proof}

By \eqref{e:PPcond}, we have a strengthening of \tref{t:stackedgraph} for the case when $k = n+1$.

\begin{corollary}\label{c:stackPP}
Let $G$ be the stack corresponding to an $\OA(n+1,n)$. Then $G$ is an
empty graph when $n \equiv 0,1 \mod{4}$ and $G$ is a
complete graph when $n \equiv 2,3 \mod{4}$.
\end{corollary}



Next we study the consequences of equation~\eref{e:PPcond} on 
$\tau$- and $\sigma$-graphs. First, we give an interpretation of this equation 
for a $\sigma$-graph. Then we use this fact to derive a property of the
$\tau$-graphs. 

\begin{lemma}\label{l:sigmagraphdegrees}
Let $A$ be an $\OA(n+1,n)$ with $\sigma$-graph $\G$.
\begin{itemize}
\item[(i)] If $n \equiv 0 \mod{4}$, then every vertex
  of $\G$ has even degree.
\item[(ii)] If $n \equiv 1 \mod{4}$, then the degrees of the vertices
  of $\G$ all have the same parity.
\item[(iii)] If $n \equiv 2 \mod{4}$, then every vertex of $G$ has
  odd in-degree and odd out-degree.
\item[(iv)] If $n \equiv 3 \mod{4}$, then the in-degrees of the
  vertices of $\G$ are all of one parity and the out-degrees of
  the vertices of $\G$ are all of the other parity.
\end{itemize}
\end{lemma}

\begin{proof}
By the handshake lemma, it suffices to show that the (in-)degrees of
the $n+1$ vertices in $\G$ all have the same parity. The parity of the
(in-)degree of vertex $i \in [1,n+1]$ is
$\pi\big(\prod_{c\neq i} \sigma_{ci}\big)$. Let $i,j \in [1,n+1]$ be two distinct
vertices. Then, in $\Z_2$,
\begin{align*}
 \pi\bigg(\prod_{c \neq i} \sigma_{ci}\bigg)+
 \pi\bigg(\prod_{c \neq j} \sigma_{cj}\bigg) 
&= 
\pi\bigg(\sigma_{ij} \sigma_{ji} \, \prod_{c\neq i,j} \sigma_{ci}\sigma_{cj}\bigg) && \\
&= \pi(\sigma_{ij} \sigma_{ji}) +  \sum\limits_{c \neq i,j} \tau^c_{ij}  && 
\mbox{by \eref{e:taufromsigmas}} \\
&= {n \choose 2} + {n \choose 2} = 0 && \mbox{by 
\eref{e:sigmacommutativity} and \eref{e:PPcond}} 
\end{align*}
Hence, the sum of the (in-)degrees of vertices $i$ and $j$ is even.
Therefore, any two vertices $i$ and $j$ have the same parity of their
(in-)degrees.
\end{proof}

Recall that every $\tau$-graph for an $\OA$ is a disjoint 
union of an isolated vertex and a complete bipartite graph. When 
we have an $\OA(n+1,n)$, the complete bipartite subgraphs in its $\tau$-graphs 
are induced on $n$ vertices. When $n$ is odd, one of the partite sets 
has an even number of vertices and the other has an odd 
number of vertices. The next theorem considers what happens when $n$ is even.

\begin{lemma}\label{l:parityofpartitesets}
Let $n$ be a positive even integer and suppose that a $\tau$-graph
of an $\OA(n+1, n)$ is the disjoint union of an isolated
vertex and $K_{n_1,n_2}$ for some nonnegative integers $n_1$ and
$n_2$. Then $n_1 \equiv n_2 \equiv n/2 \mod{2}$.
\end{lemma}

\begin{proof}
We have that $n_1+n_2=n \equiv 0 \mod{2}$ and the number of edges is $n_1n_2$. 
Hence, if there is an odd number of edges, then $n_1 \equiv n_2 \equiv 1 
\mod{2}$. Otherwise, $n_1 \equiv n_2 \equiv 0 \mod{2}$. Let $c \in [1,n+1]$. 
Working in $\Z_2$, the parity of the number of edges in the 
$\tau$-graph $G_c$ is
\begin{align*}
\sum\limits_{i<j} \tau^c_{ij} 
&= \pi \bigg(\prod _{i<j}\sigma_{ci} \sigma_{cj}\bigg) 
= (n-1)\pi\bigg(\prod_{i \ne c} \sigma_{ci}\bigg) 
= \left\{\begin{array}{ll}
          0 & \text{ if } n \equiv 0 \mod{4}, \\
          1 & \text{ if } n \equiv 2 \mod{4}, \\
         \end{array} \right.
\end{align*}
by \eref{e:taufromsigmas} and \lref{l:sigmagraphdegrees},
where $i,j \in [1,n+1] \setminus \{c\}$. 
\end{proof}

When $k=n+1$, we call a plausible $\tau$-parity that satisfies
\eqref{e:PPcond} a {\it PP-plausible $\tau$-parity}.  We next show
that only some of the plausible $\tau$-parities are PP-plausible.

\begin{theorem}\label{t:PPplausible}
The number of PP-plausible $\tau$-parities is $2^{n \choose 2}$ if $n$
is odd and $2^{ {n \choose 2}-1}$ if $n$ is even.
\end{theorem}

\begin{proof}
We choose to use \eqref{e:convent} for standardising $\sigma$-parity. 

First suppose that $n$ is even.  We can choose $\pi(\sigma_{ij})$ for
$1\le i<j\le n$, except $\pi(\sigma_{12})$.  The values of
$\pi(\sigma_{ij})$ for $1\le j<i\le n$ are then determined by
\eref{e:sigmacommutativity}.  The values of $\pi(\sigma_{ij})$ when
$n+1\in\{i,j\}$ can now be determined from \lref{l:sigmagraphdegrees}.
The $\sigma$-parity (and hence the $\tau$-parity) is thus
determined by ${n \choose 2} -1$ binary choices. 
Next we argue that each of these options for the $\tau$-parity
satisfies \eref{e:PPcond} and hence is PP-plausible. Fix $i,j\in[1,n+1]$
and let $\delta_i,\delta_j$ be, respectively, the (in-)degrees of 
vertices $i,j$ 
in the $\sigma$-graph. By construction, $\delta_i$ and $\delta_j$ agree mod 2.
So, by \eref{e:taufromsigmas} and \eref{e:sigmacommutativity}, we have 
\[
\sum_{c\in[1,n+1]\setminus\{i,j\}}\tau^c_{ij}
=\sum_{c\in[1,n+1]\setminus\{i,j\}}\pi(\sigma_{ci}\sigma_{cj})
=\delta_i+\delta_j-\pi(\sigma_{ij})-\pi(\sigma_{ji})={n\choose2}
\]
in $\Z_2$, confirming \eref{e:PPcond}.

The situation for odd $n$ is similar except that we also get to choose
$\pi(\sigma_{1(n+1)})$. Once we have chosen $\pi(\sigma_{1j})$ for $1<j\le n+1$
we know the parity of the (out-)degree of every vertex in the $\sigma$-graph,
and we can proceed as for the even $n$ case.
\end{proof}

\begin{corollary}\label{c:dimPP}
We have $\dim(n+1,n)\le {n \choose 2}$ if $n$ is odd 
and $\dim(n+1,n)\le { {n \choose 2}-1}$ if $n$ is even.
\end{corollary}

Comparing to \tref{t:sigmaplausible}, this means that not all
plausible $\tau$-parities are actual $\tau$-parities for
$(k,n)=(n+1,n)$. We have no way of judging how many PP-plausible
$\tau$-parities are actual, although it is an interesting
question.

\section{Parity of Latin squares in the ensemble}\label{s:ensemble}

In this section we primarily consider the parity of Latin squares in the
ensemble of an orthogonal array.  We will find various bounds and
congruences that must be satisfied by the number of Latin squares of a
particular parity in the ensemble.  The results will depend on the
congruence class of $n$ modulo $4$. 

While most of the section deals with properties of the ensemble, we
begin with a result that constrains the number of each parity-type
allowed among the $n-1$ Latin squares in a complete set of MOLS.

\begin{theorem}\label{t:n-1molsparitytypes}
Suppose there exist $n-1$ MOLS of order $n$ and let $z, y_1,y_2$ and
$y_3$ be the number of these Latin squares of parity types $000$,
$011$, $101$ and $110$, respectively, if $n \equiv 0,1 \mod{4}$, or of
parity types $111$, $100$, $010$ and $001$, respectively, if $n \equiv
2,3 \mod{4}$. Then
\begin{itemize}
\item[(i)] $z+y_1+y_2+y_3 = n-1$, and 
\item[(ii)] if $n$ is even then $y_1\= y_2\= y_3\not\= z\;\mod2$;\\
if $n\= 1\;\mod4$ then $y_1\= y_2$ and $y_3\= z\;\mod2$;\\
if $n\= 3\;\mod4$ then $y_1\not\= y_2$ and $y_3\not\= z\;\mod2$.
\end{itemize}
Moreover, if $z,y_1,y_2,y_3 \in [0,n-1]$ satisfy {\it(i)} and
{\it(ii)} then there is a PP-plausible $\tau$-parity 
corresponding to $n-1$ MOLS of order $n$ which include
$z,y_1,y_2$ and $y_3$ Latin squares of the appropriate parity types.
\end{theorem}

\begin{proof}
We choose to use \eqref{e:convent} for standardising $\sigma$-parity. 

Let $M = \{M_1, \dots, M_{n-1}\}$ be a set of MOLS of order $n$ and let
$A=\A(M)$, which is an $\OA(n+1,n)$.  Let $z, y_1,y_2$ and $y_3$ be the
number of squares in $M$ that are of parity types $111$, $100$, $010$
and $001$, respectively if $n \equiv 2,3 \mod{4}$, or of parity types
$000$, $011$, $101$ and $110$, respectively if $n \equiv 0,1 \mod{4}$.
Clearly {\it(i)} is satisfied.

By \eqref{e:sigmacommutativity} and \eqref{e:taufromsigmas}, 
for $c\in[3,n+1]$, we have $\tau^1_{2c} = \pi(\sigma_{1c})$ and
\begin{align*}
\tau^2_{1c} = 
\begin{cases}
1 + \pi(\sigma_{2c}) & \text{ if }  n \equiv 2,3 \mod{4},\\
\pi(\sigma_{2c}) & \text{ if }  n \equiv 0,1 \mod{4}.
\end{cases}
\end{align*}
In the $\sigma$-matrix of $A$, the entries in the first two rows in 
column $c$ determine the parity type of $M_{c-2}$ for $c \in [3, n+1]$.
These entries will be, respectively,
\begin{align*}
& (0,0), \, (0,1), \, (1,0) \, \text{ or } (1,1) 
\text{ if } M_{c-2} \text{ has parity type } 
000, \,  011, \, 101 \text{ or } 110, \text{ or}  \\
& (1,0), \, (1,1), \, (0,0) \, \text{ or } (0,1) 
\text{ if } M_{c-2} \text{ has parity type } 
111, \,  100, \, 010 \text{ or } 001.
\end{align*}
Thus, if $n \equiv 0,1 \mod{4}$ the number of ones in the first row of
the $\sigma$-matrix is $y_2+y_3$ and the number of ones in the second
row is $y_1+y_3$.  If $n \equiv 2,3 \mod{4}$ the number of ones in the
first row of the $\sigma$-matrix is $z+y_1$ and the number of ones in
the second row is $1 + y_1+y_3$, since $\pi(\sigma_{21})=1$.  Now {\it(ii)}
follows from \lref{l:sigmagraphdegrees}.

%

On the other hand, suppose $z,y_1,y_2,y_3 \in [0,n-1]$ satisfy
{\it(i)} and {\it(ii)}.  Define a matrix $W =(w_{ij})$ or order $n+1$
such that $w_{12}=0$, 
and among the pairs $\{(w_{1c}, w_{2c}):c \in [3, n+1]\}$ there are,
respectively, $z$, $y_1$, $y_2$ and $y_3$ occurrences of 
$(1,0)$, $(1,1)$, $(0,0)$ and $(0,1)$,  if $n \equiv 2,3\mod{4}$ or $(0,0)$,
$(0,1)$, $(1,0)$ and $(1,1)$, if $n \equiv 0,1 \mod{4}$.
As in the proof of \tref{t:PPplausible} the remaining entries of $W$
above the main diagonal, except the last column, can be chosen to be
either $0$ or $1$; then the last column and the remaining entries
below the main diagonal can be determined such that $W$ is a
$\sigma$-matrix corresponding to a PP-plausible $\tau$-parity.
The condition $(ii)$ ensures that the first two rows of $W$ have the
same total, mod 2.
\end{proof}

\begin{corollary}\label{c:n-1alloddsquares}
There exists a PP-plausible $\tau$-parity corresponding to a complete
set of equiparity MOLS if and only if $n\not\= 3\mod4$.
\end{corollary}

We stress that \cyref{c:n-1alloddsquares} only claims that for the 
corresponding $\OA(n+1,n)$, choosing the first two columns and any 
other column will create a subarray $B$ for which $\M(B)$ is equiparity.
However, there will be Latin squares in the ensemble
(arising from other choices of $3$ columns) which are not equiparity.
For the rest of this section, we consider the number of equiparity
Latin squares in the ensemble of an $\OA(k,n)$.

Next, we derive some equations which will be used in several of our proofs.
Let $A$ be an $\OA(k,n)$ where $3 \leq k \leq n+1$ and $n \geq 2$.  Let
$T$ be the total number of edges among the $\tau$-graphs for $A$.  Let
$x$ denote the number of equiparity Latin squares in the ensemble of
$A$, so $x$ is the number of $111$ type Latin squares if $n \= 2,3\mod4$ 
or the number of $000$ type Latin squares if $n \= 0,1\mod4$.
We next relate $T$ and $x$.  If $n \= 0,1 \mod4$ then each
non-equiparity Latin square in the ensemble contributes $2$ to $T$,
while equiparity Latin squares contribute nothing.  If $n \=
2,3\mod4$ then each non-equiparity Latin square in the ensemble
contributes $1$ to $T$, while each equiparity Latin square contributes
$3$.  Thus,
\begin{equation}\label{e:totedgex}
T = \left\{ \begin{array}{ll}
2{k \choose 3} -2x & \quad \text{if } n \equiv 0,1\mod4,\\
2x +{k \choose 3} & \quad \text{if } n \equiv 2,3\mod4.\\
\end{array} \right.
\end{equation}

We can also count $T$ using the row sums of the $\sigma$-matrix $M$,
i.e.~the adjacency matrix of the $\sigma$-graph.  Let $\mu_c$ denote
the sum of the entries in row $c$ of $M$.  Observe that, by
\eqref{e:taufromsigmas}, $\tau^c_{ij}$ is $1$ if and only if exactly
one of $\pi(\sigma_{ci})$ and $\pi(\sigma_{cj})$ is $1$, for distinct integers
$c,i,j\in[1,k]$.  For a fixed integer $c$, the number of pairs of
columns such that the cells in row $c$ (not including the main
diagonal) have entries $0$ and $1$, is the number of $\tau^c_{ij}$
which are $1$; there are $\mu_c(k-1-\mu_c)$ such pairs.  Thus,
\begin{equation}\label{e:Tintermsofmu} 
T = \sum_{c=1}^{k}\mu_c(k-1-\mu_c).
\end{equation}

\medskip


Next we consider each of the congruence classes for $n$ modulo $4$
separately. When $n \equiv 0 \mod{4}$, there are many examples of
orthogonal arrays $\OA(n+1,n)$ for which every Latin square in the
ensemble is an equiparity Latin square (necessarily of type $000$).
For example, as reported in \cite{GB12}, 18 of the 22 known projective
planes of order 16 have the property that every associated
$\OA(17,16)$ has only equiparity Latin squares in its ensemble.
The only four that do not, namely MATH-D, JOHN-D, BBS4 and BBS4-D,
still have the property that {\em some} associated
$\OA(17,16)$ has only equiparity Latin squares in its ensemble.
As an aside, we noticed that every one of the 22 projective planes has 
at least one associated complete set of MOLS in which each of the 
15 Latin squares has all 48 permutations that contribute to the
row, column and symbol parity being even.
Such Latin squares are, of course, equiparity.

Clearly it is possible for all the Latin squares in an ensemble to be
equiparity. However there are restrictions on the number of
equiparity Latin squares in the ensemble.

\begin{theorem}\label{t:equiparity0mod4}
If $n \equiv 0 \mod{4}$ then the number of equiparity Latin squares in
the ensemble of an $\OA(n+1,n)$ is even and at least
$n(n+1)(n-4)/24$.
\end{theorem}

\begin{proof}
Suppose $n \equiv 0 \mod{4}$ and let $A$ be an $\OA(n+1,n)$.  Let $T$
be the total number of edges among the $\tau$-graphs of $A$, $\mu_c$
be the sum of the entries in row $c$ of the $\sigma$-matrix, and $x$
be the number of equiparity Latin squares in the ensemble of $A$.  By
\eqref{e:totedgex} and \eqref{e:Tintermsofmu},
\[
T = 2 {n+1 \choose 3} - 2x = \sum_{c=1}^{n+1}\mu_c(n-\mu_c) \= 0 \mod4,
\]
where the congruence uses the fact that, by
\lref{l:sigmagraphdegrees}, $\mu_c$ is even for each $c\in[1,n+1]$.
It follows that $x \= {n+1 \choose 3} \mod2$, and hence $x$ is even,
since $n \equiv 0 \mod4$.  Finally, we note that $x$ is minimised when
each $\mu_c=n/2$ and $T = (n+1)(n/2)^2$. So $x \geq n(n+1)(n-4)/{24}$.
\end{proof}

For $n\= 1\mod4$, the analogue of \tref{t:equiparity0mod4} is as follows,
given that the maximum possible value of $T$ is $(n+1)^2(n-1)/4$:

\begin{theorem}\label{t:equiparity1mod4}
If $n \equiv 1 \mod{4}$ then the number of equiparity Latin squares in
the ensemble of an $\OA(n+1,n)$ is at least $(n+1)(n-1)(n-3)/24$.
\end{theorem}

In the case $n \= 1\mod4$, we cannot deduce whether
the number of equiparity Latin squares is odd or even. Indeed, both
possibilities occur amongst the PP-plausible $\tau$-parities. To 
see this, consider $\sigma$-graphs that are isomorphic
to cycles of length $3$ and $4$, respectively. By
\eqref{e:totedgex} and \eqref{e:Tintermsofmu}, the difference
in the number of equiparity Latin squares between these two examples
is $n-2$, which is odd.

Another difference between the cases $n \= 0\mod4$ and $n \= 1\mod4$
can be seen by considering a Desarguesian plane $\Pi$ of order $n$ if
$n$ is a prime power.  It is well-known that the collineation group of
$\Pi$ is doubly-transitive. Given our comments in \sref{s:intro}, this
means that every set of MOLS associated with $\Pi$ is
isotopic. (Firstly, transitivity on lines ensures there is only one
isomorphism class of $\OA$ associated with $\Pi$. Secondly, as a
result of double transitivity, every ordered pair of columns of the
$\OA$ is equivalent. The first two columns define the rows and columns
of the Latin squares and the ordering of the other columns of the $\OA$
merely determines the ordering of the Latin squares, which is irrelevant
in a set of MOLS.) The standard set of MOLS associated with $\Pi$
consists exclusively of Latin squares that are isotopic to the
elementary abelian group. By the above comments, the same must be true
of {\it every} set of MOLS derived from $\Pi$. If $\Pi$ has even
order, this means that every Latin square in the ensemble must be
equiparity since the Cayley table of any group is equiparity (see
\sref{s:intro}) and parity is an isotopism invariant for even orders, by
\lref{l:tauparconj}. For odd orders the situation is different. In
\sref{ss:swclass}, we saw that the Desarguesian plane of order $9$
does not correspond to any $\OA(10,9)$ for which the ensemble consists
exclusively of equiparity Latin squares (although the other 3
projective planes of order $9$ do!).



Unlike the case when $n\= 0,1\mod4$, it is impossible to build
an $\OA(k,n)$ purely from equiparity Latin squares when $n\= 2,3\mod4$
and $k>3$.

\begin{lemma}\label{l:notallequipar}
Suppose $n \equiv 2,3 \mod{4}$ and $A$ is an $\OA(k,n)$. 
Let $B$ be any $\OA(4,n)$ formed from $4$ columns of $A$.
At most $2$ of the $4$ Latin squares in the ensemble of $B$ are
equiparity.
\end{lemma}

\begin{proof}
Without loss of generality, consider columns 1,2,3,4 and suppose the
Latin squares defined by columns $1,2,3$ and by columns $1,2,4$ are of
parity type $111$. Then $\tau_{34}^1=0=\tau_{34}^2$ by
\eqref{e:Fixedcol}, so the other two squares are not of parity type
$111$.
\end{proof}

This last result captures the intention behind Theorem 3.3 in
\cite{GB12}.  However, that Theorem was stated for Latin squares that
are group-based. Such Latin squares are necessarily equiparity, as we
noted in \sref{s:intro}. However, when $n\equiv2\mod 4$ they are
well-known to not have any orthogonal mate, making Theorem 3.3 as
stated in \cite{GB12} vacuous. However, the class of equiparity Latin
squares is far broader than the group-based Latin squares. Yet it
turns out there cannot be too many of them in large OAs:

\begin{theorem}\label{t:asyequi}
The proportion of the Latin squares in the ensemble of an $\OA(k,n)$
that are equiparity is no more than $\frac{1}{4}+o(1)$ for
$n,k\rightarrow\infty$ with $n\equiv2,3\mod4$.
\end{theorem}	
	
\begin{proof}
We can bound the number of equiparity Latin squares in an $\OA(k,n)$
using the fact that the $\tau$-graphs are all $K_{a,b}$ where
$a+b=k-1$.  The maximum number of edges in total amongst all $G_c$ is
$k \lfloor \frac{k-1}{2} \rfloor \lceil \frac{k-1}{2} \rceil$.  Thus,
by \eqref{e:totedgex}, the number of equiparity Latin squares is at
most
\begin{equation}\label{e:mostequi} 
\frac{1}{2}\left( k \lfloor (k-1)/2 \rfloor \lceil (k-1)/2 \rceil 
- {k \choose 3} \right).
\end{equation}
Hence asymptotically, the proportion of the ${k \choose 3}$ Latin
squares that can be equiparity is at most $\frac{1}{4}+O(1/k)$.
\end{proof}

We next show that for every $n\= 2,3\mod4$
the bound \eref{e:mostequi} can be achieved 
by some PP-plausible $\tau$-parity.

\begin{lemma} 
For each $n \equiv 2,3 \mod{4}$, there exists a PP-plausible $\tau$-parity
of an $\OA(n+1,n)$ such that each $\tau$-graph is isomorphic to 
$K_1\cup K_{\lfloor{n/2} \rfloor, \lceil n/2 \rceil}$.  
\end{lemma}  

\begin{proof}
In this proof, for every integer $\xi$ we let $(\xi)_n$ denote the
lowest non-negative integer that is congruent to $\xi$ mod $n$.
For distinct integers $i,j \in [1,n+1]$, define
\[
\pi(\sigma_{ij}) = \left\{ 
\begin{array}{ll}
0  & \text{if } (j-i)_n \in [1, \lfloor n/2 \rfloor], \\
1 & \text{otherwise.} 
\end{array} \right.
\]
By \tref{t:sigmaplausible} and \lref{l:sigmagraphdegrees}, this
corresponds to a PP-plausible $\tau$-parity.  In particular, by
\eqref{e:taufromsigmas}, for distinct integers $c,i,j\in[1,n+1]$,
\[
\tau^c_{ij} = \pi(\sigma_{ci}\sigma_{cj}) = \left\{ 
\begin{array}{ll}
1  & \text{if precisely one of }(i-c)_n\text{ and }(j-c)_n\text{ is in }
[1,\lfloor n/2 \rfloor],\\
0 & \text{otherwise.} 
\end{array} \right.
\]
Thus, for each $c \in [1,n+1]$, the $\tau$-graph $G_c$ is the disjoint
union of an isolated vertex and a complete bipartite graph with
partite sets of size $\lfloor \frac{n}{2} \rfloor $ and 
$\lceil\frac{n}{2} \rceil $.
\end{proof}

We have just considered the maximum value for the number of equiparity
Latin squares in the ensemble of an $\OA(n+1,n)$. In
the next two theorems, we derive a constraint in the form of a
congruence and then a lower bound on this number.  For the 
$n\= 3\mod4$ case in both results we invoke a form of standardisation
different from \eref{e:convent}. It will be more convenient for us to
choose the parity of the out-degrees of the vertices in the
$\sigma$-graph $\G$.  By \lref{l:Sym}, $\G$ is a tournament
and by \lref{l:sigmagraphdegrees}, all out-degrees
have the same parity.  As $n+1$ is even,
$\sigma$-complementation changes the parity of all out-degrees, so we
can use it to select the parity that we want.

\begin{theorem}\label{t:equiparity23mod4}
If $n \equiv 2,3 \mod{4}$ then the number of equiparity Latin squares
in the ensemble of an $\OA(n+1,n)$ is congruent to 
$\lceil n/4 \rceil\,\mod{4}$.
\end{theorem}

\begin{proof}
Let $A$ be an $\OA(n+1,n)$. Let $x$ denote the number of equiparity
Latin squares in the ensemble of $A$ and let $\mu_c$ denote the sum of
the entries in row $c$ of the adjacency matrix of the $\sigma$-graph
$\G$.

By \lref{l:Sym}, $\G$ is a tournament, and by
\lref{l:sigmagraphdegrees}, $\mu_c$ is odd for each $c\in[1,n+1]$ if
$n \equiv 2 \mod{4}$.  If $n \equiv 3 \mod{4}$, then
by standardisation we may assume that the $\mu_c$ are all odd.
Now, by \eqref{e:totedgex} and \eqref{e:Tintermsofmu},
\begin{align}
2x + {n+1 \choose 3}  
&= \sum_{c=1}^{n+1}\mu_c(n-\mu_c) \nonumber  \\
&= n {n+1 \choose 2} - \sum_{c=1}^{n+1}\mu_c^2  \hspace{3cm}  
\textrm{ as } \G \text{ is a tournament} \label{e:totedges}\\
&\equiv n {n+1 \choose 2} - (n+1) \quad \mod{8}, \nonumber
\end{align}
since each $\mu_c$ is odd. Hence, we find that
\[
x \= \left\{
\begin{array}{ll}
 \dfrac{(2n-3)(n+1)}{3}\cdot \dfrac{n+2}{4} \= \dfrac{n+2}{4}\mod4 & \text{ if } n \= 2 \mod4, \\[3ex]
  \dfrac{(2n-3)(n+2)}{3}\cdot \dfrac{n+1}{4} \= \dfrac{n+1}{4}\mod4 & \text{ if } n \= 3 \mod4. \\
 \end{array} \right.
\]
\end{proof}

We define a sequence $\mu_1,\dots,\mu_{n+1}$ to be
{\em good} if
\begin{equation}\label{e:rowsumbnd}
\sum_{i=1}^m\mu_i\le nm-\binom{m}{2}
\end{equation}
for each $m\in[1,n+1]$. For our next theorem we will need the 
following property of good sequences:

\begin{lemma}\label{l:good}
If $\mu_1,\dots,\mu_{n+1}$ is a good sequence and $\mu_{c}=\mu_{c+1}-2$
for some $c\in[1,n]$ then the sequence obtained by interchanging
$\mu_c$ with $\mu_{c+1}$ is also good.
\end{lemma}

\begin{proof}
Consider a hypothetical counterexample to the claim.
Let $K=\sum_{i=1}^{c-1}\mu_c$. 
Applying \eref{e:rowsumbnd} for $m\in\{c-1,c,c+1\}$, we see that
\begin{align}
K&\le n(c-1)-\binom{c-1}{2},\label{e:pre}\\
K+\mu_c&\le nc-\binom{c}{2}<K+\mu_c+2,\label{e:on}\\
K+2\mu_c+2&\le n(c+1)-\binom{c+1}{2}.\label{e:post}
\end{align}
It follows from \eref{e:on} that $K+\mu_c=nc-\binom{c}{2}-\eps$ 
where $\eps\in\{0,1\}$. Subtracting this equation from \eref{e:post}
gives $\mu_c+2\le n-c+\eps$. Then, using \eref{e:on} again, we get
$K> nc-\binom{c}{2}-n+c-\eps=n(c-1)-\binom{c-1}{2}+1-\eps$,
contradicting \eref{e:pre}. 
\end{proof}

We can now derive a lower bound on the number of equiparity Latin squares
in the ensemble.

\begin{theorem}\label{t:linmanyequipar}
Let $n\equiv2,3\mod4$.  Any $\OA(n+1,n)$ has at least $\lceil n/4\rceil$
equiparity Latin squares in its ensemble. 
\end{theorem}

\begin{proof}
Let $A$ be an $\OA(n+1,n)$ and let $M$ be the $\sigma$-matrix of $A$.
Let $\mu_i$ denote the total of the entries in row $i$ of $M$.  In the
$n\equiv2\mod4$ case, each $\mu_i$ is odd by
\lref{l:sigmagraphdegrees}.  In the $n\equiv3\mod4$ case, by
standardisation we can ensure that each $\mu_i$ is even. In
either case, we have
\begin{equation}\label{e:mun-1}
\mu_i\equiv n-1\quad\mod2
\end{equation}
for each $i$.
Let $T$ be the total number of edges among the $\tau$-graphs for $A$.
The number of equiparity Latin squares in the ensemble of $A$ is
$\big(T-{n+1\choose 3}\big)/2$ by \eref{e:totedgex}, which will be
minimised by minimising $T$.  By \eref{e:totedges}, this is achieved
by maximising $\sum\mu_i^2$.  We may assume that
$\mu_1\ge\mu_2\ge\cdots\ge\mu_{n+1}$ by relabelling if necessary.
Also, \eref{e:rowsumbnd} holds for each $m$
because the leading principal minor of order $m$ in $M$ 
has half its off-diagonal entries equal to zero.

Assume now that $\mu_1,\dots,\mu_{n+1}$ is the monotonic good sequence
that maximises $\sum\mu_i^2$ subject to \eref{e:mun-1}. 
We claim that the sequence 
can be found with a greedy algorithm, in the sense that
\begin{equation}\label{e:greedy}
\sum_{i=1}^m\mu_i\ge nm-\binom{m}{2}-1
\end{equation}
for $1\le m\le n+1$. Suppose not, and let $m=a$ be such that this
inequality fails. 
Consider a new sequence $\mu'_1,\dots,\mu'_{n+1}$
for which $\mu'_a=\mu_a+2$, $\mu'_{a+1}=\mu_{a+1}-2$, and $\mu'_i=\mu_i$ for
$i\notin\{a,a+1\}$. Clearly this new sequence is good, 
satisfies \eref{e:mun-1} and can be made monotonic by repeated application
of \lref{l:good}. However,
\[
\sum_{i=1}^m(\mu'_i)^2=\sum_{i=1}^m\mu^2_i+4(\mu_a-\mu_{a+1})+8>\sum_{i=1}^m\mu^2_i,
\]
contradicting the optimality of our original sequence. 
This proves our claim that the
optimal sequence satisfies \eref{e:greedy}.

Note that there can be at most one monotonic good sequence 
satisfying \eref{e:mun-1} and \eref{e:greedy}, because 
the constraints are such that each term of the sequence is
determined by its predecessors.
We next argue that this optimal sequence is given by
$\mu_1=\mu_2=\mu_3=n-1$, $\mu_4=n-3$ and $\mu_i=\mu_{i-4}-4$ for $4<i\le n+1$.
Certainly this is a monotonic sequence of non-negative integers, since
$\mu_{n+1}=n-1-4(n-2)/4=1$ for $n\= 2\mod4$ and
$\mu_{n+1}=n-3-4(n-3)/4=0$ for $n\= 3\mod4$. Also, for $m=4\alpha+\beta$ with 
$\beta\in\{1,2,3\}$, we have
\begin{align*}
\sum_{\smash{i=1}}^m\mu_i&=4(n-1+n-5+\cdots+n-4\alpha+3)
-2\alpha +\beta(n-4\alpha-1)\\
&=4\alpha(n-2\alpha+1)-2\alpha +\beta(n-1-4\alpha)
=nm-\binom{m}{2}+\binom{\beta}{2}-\beta.
\end{align*}
Since $\binom{\beta}{2}-\beta\in\{-1,0\}$, we see that both
\eref{e:rowsumbnd} and \eref{e:greedy} are satisfied in this case.
These constraints are also satisfied in the case $m=4\alpha$, because
\begin{align*}
\sum_{i=1}^m\mu_i&=4(n-1+n-5+\cdots+n-4\alpha+3)-2\alpha
=4\alpha(n-2\alpha+1)-2\alpha
=nm-\binom{m}{2}.
\end{align*}
Returning to \eref{e:totedges}, we now know that
\begin{align*}
T&\ge n{n+1\choose 2}-3\sum_{i=0}^{\lfloor(n-1)/4\rfloor}(n-1-4i)^2
-\sum_{i=0}^{\lfloor(n-3)/4\rfloor}(n-3-4i)^2\\
&=
\begin{cases}
n^3/6+n/3+1&\text{if }n\= 2 \mod4,\\
n^3/6+n/3+1/2&\text{if }n\= 3 \mod4,\\
\end{cases}
\end{align*}
which then implies the result, via \eref{e:totedges} and \eref{e:totedgex}.
\end{proof}

\begin{corollary}\label{cy:atleast2noniso}
If $n \equiv 2 \mod{4}$ and $n>2$ then the ensemble of an $\OA(n+1,n)$
contains at least two Latin squares that are not isotopic to each other.
\end{corollary}

\begin{proof}
By \tref{t:linmanyequipar} the ensemble contains an equiparity Latin
square.  However, given that $n>2$, \lref{l:notallequipar} tells us that
the ensemble also contains a Latin square which is not
equiparity. Parity is an isotopism invariant for even orders.
\end{proof}

In our definition of ensemble we have insisted that each set of three
columns produces a single Latin square. If we had instead allowed the
three columns to be used in any order then 6 potentially different
Latin squares would be produced for each set of three
columns. \cyref{cy:atleast2noniso} could then be strengthened to say
that there will be at least 4 isotopism classes among the Latin
squares corresponding to an $\OA(n+1,n)$ when $n\equiv2\mod4$, since
each of the parity types 001,010,100,111 must occur.

It is also worth noting that there is a PP-plausible $\tau$-parity
that achieves the bound in \tref{t:linmanyequipar}.
We can build the adjacency matrix $M$ for the corresponding
$\sigma$-graph as follows.
Let
\[
B_3=\left[
\begin{array}{ccc}
0&0&1\\
1&0&0\\
0&1&0\\
\end{array}
\right]
\mbox{ \ and \ }
B_4=\left[
\begin{array}{cccc}
0&0&1&1\\
1&0&0&1\\
0&1&0&1\\
0&0&0&0\\
\end{array}
\right].
\]
If $n=4\alpha+2$, place $\alpha$ copies of the block $B_4$ and one
copy of $B_3$ down the diagonal of $M$; if $n=4\alpha+3$, place
$\alpha+1$ copies of $B_4$ down the diagonal of $M$. Fill all other
entries above the diagonal with 1's and all those below with 0's. It
is routine to check that this construction has the required
properties, including achieving the optimal sequence $\{\mu_i\}$.

In contrast to \tref{t:linmanyequipar}, 
there are plausible $\tau$-parities for which the ensemble
contains no equiparity Latin squares.

\begin{lemma}\label{l:zeroallodd}
If $n \equiv 2,3 \mod{4}$ and $k$ is arbitrary,
then there exists a plausible $\tau$-parity
for which there are no equiparity Latin squares in the ensemble of the
associated $\OA(k,n)$.
\end{lemma}

\begin{proof}  
Take the $\sigma$-matrix to be lower triangular.  Then the
$\tau$-parities (for $i<j$) are given by:
$$ \tau_{ij}^c =  \left\{ \begin{array}{ll}
0 & \textrm{if } i<j<c \textrm{ or } c<i<j, \\
1 & \textrm{if } i<c<j.
\end{array} \right.$$
For any three distinct integers $c,i,j
\in [1, k]$, exactly one of $\tau_{jc}^i, \tau_{ic}^j, \tau_{ij}^c$ 
is $1$, and the result follows.
\end{proof}

We stress that the plausible $\tau$-parities in \lref{l:zeroallodd}
are not PP-plausible, by \lref{l:sigmagraphdegrees}.  Most, though not
all, of the restrictions that we have demonstrated on the parities of
Latin squares in the ensemble only apply to the projective plane case.

\section{Concluding remarks}

We have considered two notions of parity for the orthogonal arrays
that correspond to MOLS. One of these, $\sigma$-parity, was introduced
by Glynn and Byatt \cite{GB12} in a more limited setting. The other,
$\tau$-parity, is a direct generalisation of the row, column and
symbol parities of Latin squares. These two notions of parity are
closely related; $\tau$-parity is determined by $\sigma$-parity, and
the converse is true up to complementation (see
\sref{ss:equivalencetausigma}).  The relationship between the two
parities proved very fruitful throughout our investigations. For
example, it provided a very simple proof of \lref{l:LStriangleparity},
which generalises the well-known relationship between the row, column
and symbol parities of Latin squares.

In \sref{s:graphs} we introduced useful graph theoretic models for the
two notions of parity. Each $\OA(k,n)$ has one $\sigma$-graph, $k$
$\tau$-graphs and one ``stack'' which is formed by merging the
$\tau$-graphs, modulo $2$. The $\sigma$-graph is undirected for $n\=
0,1\mod 4$, but is a tournament for $n\= 2,3\mod 4$ (\lref{l:Sym}).
The $\tau$-graphs and stack are highly structured
(\lref{l:taugraphbipartite}, \tref{t:stackedgraph}). Further
restrictions apply to the structure of $\sigma$-graphs, $\tau$-graphs
and the stack for $\OA$s that come from projective planes
(\lref{l:sigmagraphdegrees}, \lref{l:parityofpartitesets},
\cyref{c:stackPP}). These restrictions are based on the equivalence
between projective planes and sharply 2-transitive sets of
permutations.

In \sref{s:numpar} we considered the question of how many different
$\tau$-parities might be obtained by an $\OA(k,n)$. We phrased this in terms
of $\dim(k,n)$, the information content of the $\tau$-parity, in bits.
We showed that $\dim(k,n)\le{k\choose 2}-1$ (\cyref{c:dim}).  Later we
showed that $\dim(n+1,n)\le {n \choose 2}$ if $n$ is odd and
$\dim(n+1,n)\le { {n \choose 2}-1}$ if $n$ is even  (\cyref{c:dimPP}).

An interesting question is whether there are restrictions on parities
of an $\OA(k,n)$ other than the ones that we have demonstrated.  This
question is wide open for the case when $k$ is comparable in size to
$n$. However, we conjecture that when $n$ is large relative to $k$, no
further restrictions apply (\cjref{cj:allactual}). We proved this
conjecture for $k\le5$ in \sref{ss:plausisactual} by showing that all
plausible $\tau$-parities are actually achieved. Moreover, it is
possible to embed any $k$-MOLS of order $n$ inside some set of
$k$-MOLS of order $N$ whenever $N$ is large enough. Hence if $(k,n)$ is
such that $\dim(k,n)$ achieves its upper bound then
$\dim(k,N)=\dim(k,n)$ for all large $N$.

An open question, related to the issue of whether there are further
constraints when $k$ is comparable to $n$, is the issue of unique
completion. It was shown by Metsch \cite{Met91} that any
$\OA(n-O(n^{1/3}),n)$ has a unique completion to an $\OA(n+1,n)$, up
to isomorphism. It would be interesting to find parity analogues of
this result. Examining the proof of \tref{t:PPplausible}, we see that
the $\sigma$-parity of an $\OA(n,n)$ determines the $\sigma$-parity of
its completion to an $\OA(n+1,n)$, but the same is not true for an
$\OA(n-1,n)$ without further developments in the theory.

In \sref{s:ensemble} we considered the {\em ensemble} of an
$\OA(k,n)$, which is the set obtained by taking 3 columns of the $\OA$
at a time and interpreting the result as a Latin square.  It turns out
that there are many bounds and congruences that restrict the number of
Latin squares of each parity that may occur within the ensemble.  The
restrictions, which mostly apply when $k=n+1$, are weakest for 
$n\= 0,1\mod 4$.  Many of the results in \sref{s:ensemble} that apply to
$n\= 2\mod 4$ also apply to $n\= 3\mod 4$, but are not so limiting in
that case since parity is not an isotopism invariant for odd $n$.
Hence, our constraints are strictest when $n\= 2\mod 4$, perhaps
offering some insight into why projective planes of these orders are
hard to construct (and are believed by many not to exist for $n>2$).
In \cyref{cy:atleast2noniso} we showed that all the Latin squares in
the ensemble cannot be isotopic to each other when $n\= 2\mod4$.
Indeed, as $n$ grows the number of ``equiparity'' Latin squares in the
ensemble grows at least linearly (\tref{t:equiparity23mod4}), but it
can never be more than $\frac14+o(1)$ of all the Latin squares in the
ensemble (\tref{t:asyequi}).


The preponderance of even parities among the known projective planes
of order 16 invites further investigation. Certainly, for $n\= 0\mod4$
all of the constraints that we have found are trivially satisfied when
all parities are even. This means the $\sigma$-graph, all
$\tau$-graphs and the stack are all empty graphs, which seems to be
the easiest way to satisfy all known requirements. In contrast, for
$n\= 2\mod 4$ all these graphs are forced to have edges, and the
constraints such as \lref{l:parityofpartitesets} and \cyref{c:stackPP}
seem at first glance to be much harder to satisfy. However, it is
worth stressing that in \tref{t:PPplausible} we have a mechanism for
producing many choices for the parities that satisfy all the
constraints that we have demonstrated in this paper. Our work cannot
be used to rule out the existence of any projective plane without the
discovery of a new constraint on its parity.

\subsection*{Acknowledgements}

The authors are very grateful to Peter Dukes for supplying
\tref{t:asymIMOLSexistence} and to Darcy Best for helpful proofreading.

  \let\oldthebibliography=\thebibliography
  \let\endoldthebibliography=\endthebibliography
  \renewenvironment{thebibliography}[1]{%
    \begin{oldthebibliography}{#1}%
      \setlength{\parskip}{0.4ex plus 0.1ex minus 0.1ex}%
      \setlength{\itemsep}{0.4ex plus 0.1ex minus 0.1ex}%
  }%
  {%
    \end{oldthebibliography}%
  }


\begin{thebibliography}{99}
\def\baselinestretch{1}\small\normalsize
\begin{small}

\bibitem{AL15}
R.~Aharoni, M.~Loebl, 
The odd case of Rota's bases conjecture,
{\em Adv. Math.} {\bf282} (2015), 427--442.

\bibitem{Alp17}
L.~Alpoge,
Square-root cancellation for the signs of Latin squares,
{\it Combinatorica}, to appear. DOI: 10.1007/s00493-015-3373-7.


\bibitem{BCN89} A.\,E.~Brouwer, A.\,M.~Cohen and A.~Neumaier, 
{\it Distance-Regular Graphs}, Springer, Berlin, 1989. 

\bibitem{CW16}
N.\,J.~Cavenagh and I.\,M.~Wanless, 
There are asymptotically the same number of Latin squares of each parity,
{\em Bull. Aust. Math. Soc.} {\bf94} (2016), 187--194.

\bibitem{CD11} C.\,J.~Colbourn and J.\,H.~Dinitz, 
Mutually orthogonal Latin squares: A brief survey of constructions,
{\it J. Statist. Plann. Inference}, {\bf 95} (2001), 9--48.

\bibitem{handbook} 
C.\,J.~Colbourn and J.\,H.~Dinitz (eds),
{\it Handbook of combinatorial designs}, 2nd ed.,
Chapman \& Hall/CRC, Boca Raton, FL, 2007. xxii+984 pp.

\bibitem{DGGL10} D.\,M.~Donovan, M.\,J.~Grannell, T.\,S.Griggs and 
J.\,G.~Lefevre, On parity vectors of Latin squares, {\it Graphs Combin.} 
{\bf26} (2010) 673--684.

\bibitem{DvB15} P.\,J.~Dukes and C.\,M.~van Bommel, 
Mutually orthogonal Latin squares with large holes, 
{\it J. Statist. Plann. Inference}, {\bf 159} (2015), 81--89.

\bibitem{EW16}
J.~Egan and I.\,M.~Wanless, 
Enumeration of MOLS of small order,
{\it Math. Comp.} {\bf85} (2016), 799--824.

\bibitem{Gly10} D.\,G.~Glynn,
The conjectures of Alon-Tarsi and Rota in dimension prime minus one,
{\it SIAM J. Discrete Math.} {\bf24} (2010), 394--399. 

\bibitem{GB12} D.\,G.~Glynn and D.~Byatt, 
Graphs for orthogonal arrays and projective planes of even order,  
{\it SIAM J. Discrete Math.} {\bf 26} (2012), No. 3, 1076--1087.

\bibitem{Jac06} E.~Jacobsthal, Anwendungen eiuer formel aus der theorie der 
quadratischen reste, Dissertation (Berlin, 1906), 26--32.

\bibitem{Jan95} J.\,C.\,M.~Janssen, 
On even and odd Latin squares, 
{\it J. Combin. Theory Ser. A}, {\bf 69} (1995) 173--181.

\bibitem{KMOW14}
P.~Kaski, A.\,D.\,S. Medeiros, P.\,R.\,J.~\"Osterg\aa rd and I.\,M.~Wanless, 
Switching in one-factorisations of complete graphs,
{\it Electron. J. Comb.\/} {\bf21}(2) (2014), \#P2.49.

\bibitem{KD15}
A.\,D.~Keedwell and J.~D\'enes, {\it Latin squares and their applications}
(2nd ed.),
North Holland, Amsterdam, 2015.

\bibitem{Kot12}
D.~Kotlar,
Parity types, cycle structures and autotopisms of Latin squares,
{\em Electron. J. Combin.} {\bf19}(3) (2012), \#P10, 17 pp.

\bibitem{MS13} B.\,D.~McKay and P.\,~Schweitzer, 
Switching reconstruction of digraphs,   
{\it J. Graph Theory} {\bf76} (2014), 279--296.

\bibitem{Met91}
K.~Metsch, 
Improvement of Bruck's completion theorem,
{\it Des. Codes Cryptogr.} 1 (1991), 99--116. 

\bibitem{OEIS}
N.\,J.\,A.~Sloane, The On-Line Encyclopedia of
Integer Sequences,\\ 
\url{http://www.research.att.com/~njas/sequences/}.

\bibitem{vB15} C.\,M.~van Bommel, 
An asymptotic existence theory on incomplete mutually orthogonal Latin squares, 
Master's thesis, Uni. Victoria, (2015).

\bibitem{SW12}
D.\,S.~Stones and I.\,M.~Wanless, 
How not to prove the Alon-Tarsi conjecture,
{\it Nagoya Math.\ J.} {\bf205} (2012), 1--24.


\bibitem{Wan04} I.\,M.~Wanless, 
Cycle switches in Latin squares, 
{\it Graphs Combin.} {\bf20} (2004) 545--570.

\bibitem{WWWW} I.\,M.~Wanless, 
Data on Mutually Orthogonal Latin Squares (MOLS)\\  
\url{http://users.monash.edu.au/~iwanless/data/MOLS/}.

\bibitem{Zappa96} P.\,Zappa, Triplets of Latin squares, 
{\it Bollettino U.M.I.} {\bf (7) 10-A} (1996) 63--69.

\end{small}
\end{thebibliography}
\end{document}